\documentclass[]{article}

\usepackage{amsfonts}
\usepackage{amsmath}
\usepackage{amssymb}
\usepackage{amsthm}
\usepackage{indentfirst}
\usepackage[hidelinks]{hyperref}
\usepackage{url}

\numberwithin{equation}{section}
\theoremstyle{plain}
\newtheorem{theorem}{Theorem}[section]
\newtheorem{lemma}[theorem]{Lemma}
\newtheorem{proposition}[theorem]{Proposition}
\theoremstyle{remark}

\newcommand{\R}{\mathbb{R}}
\newcommand{\Mr}{\mathcal{M}_r}
\newcommand{\F}{\mathrm F}
\newcommand{\norm}[1]{\left\lVert #1\right\rVert}
\newcommand{\ip}[2]{\left\langle #1,#2\right\rangle_{\F}}
\newcommand{\eps}{\varepsilon}
\newcommand{\rank}{\operatorname{rank}}

\hypersetup{
 pdftitle={A Counterexample to Robust Second-Order Convergence of the Strang Projector-Splitting Integrator},
 pdfauthor={Shiheng Zhang},
 pdfkeywords={dynamical low-rank approximation; projector-splitting integrator; Strang splitting; robust error bounds; small singular values; counterexample}
}

\begin{document}

\title{A Counterexample to Robust Second-Order Convergence of the Strang
Projector-Splitting Integrator}
\author{Shiheng Zhang\thanks{Department of Applied Mathematics,
University of Washington, Seattle, WA 98195, USA
(\texttt{shzhang3@uw.edu}).}}
\date{}

\maketitle

\begin{abstract}
The classical Strang projector-splitting integrator is widely observed
to converge with order two, whereas the error analysis that remains
uniform as the smallest singular value retained in the low-rank
approximation tends to zero proves only order one.  We show that this
gap is intrinsic under the standard assumptions.  We construct
$3\times3$ matrix differential equations that are $C^2$ in time and
smooth in the matrix variable, with rank-two initial data
that satisfy uniform boundedness, Lipschitz, tangency-defect, and
regularity bounds.  Nevertheless, the exact-subflow Strang method has
a nonzero $h^2$ term in the local error over one periodic forcing cycle
consisting of four Strang steps.  Repetition of that cycle rules out a
global second-order bound
whose constant and stepsize threshold are independent of the retained
singular values.  The mechanism is a rapid rotation of the factor
directions associated with the small singular value: first-order
consistency is preserved, but the changing projection spaces prevent
the cancellation normally expected from a symmetric Strang
composition.  Hence the robust first-order result cannot, under these
assumptions alone, be upgraded to robust second order for the classical
projector-splitting method.
\end{abstract}



\section{Introduction}
\label{sec:introduction}

Consider a matrix differential equation
\begin{equation}\label{eq:ode-intro}
 \dot A(t)=F(t,A(t)),\qquad A(0)=A^0.
\end{equation}
Dynamical low-rank approximation replaces $A(t)$ by a matrix of fixed
rank $r$ and projects $F(t,Y)$ onto the corresponding tangent
space \cite{KochLubich2007}.  The projector-splitting integrator carries
out this projection through equations for the factors in
$Y=USV^\top$ \cite{LubichOseledets2014}.  Kieri, Lubich, and Walach
proved that both its Lie and Strang variants satisfy a global bound of the
form
\[
 \norm{Y^n-A(t_n)}_{\F}\le C(\delta+\eps+h),
\]
where $\delta$ is the initial error, $\eps$ bounds the component normal
to the rank-$r$ manifold, and the constant is independent of the
retained singular values
\cite{KieriLubichWalach2016}.  This is the robust first-order result on
which the present paper builds.
Following the standard terminology, an error bound is called
\emph{robust} if its constants and stepsize threshold remain independent
of the retained singular values as these approach zero.

Second-order convergence of the Strang projector-splitting integrator
is nevertheless widely observed numerically, while the known robust
proof yields only order one
\cite{CerutiEinkemmerKuschLubich2024}.  Robust second-order convergence
has instead been proved for midpoint, parallel, and Runge--Kutta
basis-update-and-Galerkin (BUG) integrators
\cite{CerutiEinkemmerKuschLubich2024,Kusch2025,NobileRiffaud2026},
which are different numerical methods.  These facts leave two possible explanations:
either the analysis loses an additional order, or robust second-order
convergence fails.  The present paper proves the
second alternative for the assumptions in Section~\ref{sec:method},
even when all five factor subproblems are solved exactly.  To our
knowledge, this is the first analytic counterexample to a
retained-singular-value-independent second-order bound for the classical
exact-subflow Strang projector-splitting integrator under these
assumptions.

The counterexample also explains the failure mechanism.  Near a small
retained singular value, a perturbation that is small in matrix norm can
produce an order-one change in the associated factor spaces.  This
destroys the second-order cancellation normally expected from Strang
symmetry while preserving first-order consistency.  Repetition then
accumulates the resulting local error, showing that robust second-order
convergence fails under the stated assumptions.

Section~\ref{sec:method} recalls the assumptions and the five exact
subflows.  Section~\ref{sec:construction} defines
the two equations in the counterexample family;
Section~\ref{sec:defect} computes
the nonzero four-step term, and Section~\ref{sec:accumulation} proves the
theorem.  The rank and projector estimates are stated in
Section~\ref{sec:defect}.  Appendix~A supplies the desingularized factor
equations, rank margins, and uniform projector expansions used in the
proof.

\section{Assumptions and the exact Strang step}
\label{sec:method}

Let
\[
 \Mr:=\{Y\in\R^{m\times n}:\rank(Y)=r\}.
\]
For $Y=USV^\top\in\Mr$, with $U,V$ column-orthonormal and $S$
nonsingular, its tangent space and orthogonal tangent projector are
\[
 T_Y\Mr=\{Z:(I-UU^\top)Z(I-VV^\top)=0\},
\]
and
\begin{equation}\label{eq:tangent-projector}
 \Pi(Y)Z:=ZVV^\top-UU^\top ZVV^\top+UU^\top Z.
\end{equation}
We use the Frobenius norm for matrices.  Operator norms are written
$\|\cdot\|_{\F\to\F}$,
$\|\cdot\|_{(\F,\F)\to\F}$, according to their domain and range; for
a bilinear map $\mathcal T$,
\[
 \|\mathcal T\|_{(\F,\F)\to\F}
 :=\sup_{\norm X_{\F}=\norm Z_{\F}=1}\norm{\mathcal T[X,Z]}_{\F}.
\]
Throughout, $A(t)$ denotes the full solution, $Y$ a rank-$r$
approximation, and $X,Z$ generic matrix arguments or temporary
subproblem states.

Fix $Y^0\in\Mr$, $\delta,\eps\ge0$, and constants
$(B_F,L_F,T,\kappa)$.  The robust analysis assumes
\begin{equation}\label{eq:basic-bounds}
 \norm{F(t,X)}_{\F}\le B_F,\qquad
 \norm{F(t,X)-F(t,Z)}_{\F}\le L_F\norm{X-Z}_{\F}.
\end{equation}
Moreover,
\begin{align}
 F(t,Y)&=M(t,Y)+R(t,Y),\qquad M(t,Y)\in T_Y\Mr,\notag\\
 \norm{M(t,Y)}_{\F}&\le B_F,\qquad \norm{R(t,Y)}_{\F}\le\eps,
 \label{eq:tangent-decomp}\\
 \norm{Y^0-A^0}_{\F}&\le\delta.
 \label{eq:remainder-initial}
\end{align}
The counterexample will also have the fixed smoothness bounds
\begin{equation}\label{eq:smooth-bounds}
 \norm{D_X^2F}_{(\F,\F)\to\F},\quad
 \norm{\partial_tF}_{\F},\quad
 \norm{\partial_t^2F}_{\F},\quad
 \norm{D_X\partial_tF}_{\F\to\F}
 \le\kappa.
\end{equation}
All bounds hold for $t\in[0,T]$ and $X,Z\in\R^{m\times n}$; the tangent
decomposition is required for $Y\in\Mr$.

We next fix the numerical method.  Start one step at time $t$ from
$Y=U_0S_0V_0^\top$.  The first K subproblem keeps $V_0$ fixed and
its endpoint determines a new column basis $U_1$.  The L subproblem
then determines a new orthonormal basis $V_1$ for the row space.  Define
\begin{equation}\label{eq:projectors-one-step}
 P_0:=V_0V_0^\top,\qquad Q:=U_1U_1^\top,
 \qquad P_1:=V_1V_1^\top.
\end{equation}
One exact-subflow Strang step is the following ordered composition; each
subproblem uses the indicated time interval.  In the table, \(t'\)
is the running time in that interval, and the initial value of each
subproblem is the factor data produced by the preceding row.
\begin{equation}\label{eq:five-substeps}
\begin{array}{c|c|c}
\text{substep}&\text{time interval}&\text{matrix equation}\\
\hline
K_1&[t,t+h/2]&\dot X=F(t',X)P_0\\
S_1&[t,t+h/2]&\dot X=-QF(t',X)P_0\\
L  &[t,t+h]  &\dot X=QF(t',X)\\
S_2&[t+h/2,t+h]&\dot X=-QF(t',X)P_1\\
K_2&[t+h/2,t+h]&\dot X=F(t',X)P_1.
\end{array}
\end{equation}
These are factor equations: for example, the first is
\[
 \dot K=F(t',KV_0^\top)V_0,\qquad K(t)=U_0S_0,
\]
with $X=KV_0^\top$.  In the five rows of
\eqref{eq:five-substeps}, respectively, the temporary matrix is
\begin{equation}\label{eq:five-factor-representations}
 X=K_1V_0^\top,\quad U_1S_1V_0^\top,\quad U_1L^\top,
 \quad U_1S_2V_1^\top,\quad K_2V_1^\top.
\end{equation}
The K and L endpoints are QR-factorized when a
new basis is needed; the S subproblems update the $r\times r$ core.
The five factor problems are connected as follows:
\begin{equation}\label{eq:factor-transfers}
\begin{aligned}
 K_1(t)&=U_0S_0,
 &K_1(t+h/2)&=U_1\widehat S_1,\\
 S_1(t)&=\widehat S_1,
 &S_1(t+h/2)&=\widetilde S_1,\\
 L(t)&=V_0\widetilde S_1^\top,
 &L(t+h)&=V_1\widehat S_2^\top,\\
 S_2(t+h/2)&=\widehat S_2,
 &S_2(t+h)&=\widetilde S_2,\\
 K_2(t+h/2)&=U_1\widetilde S_2,
 &K_2(t+h)&=U_2S^+.
\end{aligned}
\end{equation}
Here the first, third, and fifth endpoint identities are QR
factorizations, and the final matrix is
$U_2S^+V_1^\top$.  Thus every subproblem is run from the factor data
produced by the preceding one.  Equation \eqref{eq:factor-transfers}
specifies both the transferred factor data and the time interval used by
each subproblem.
For orthogonal $G_U,G_V\in\mathbb R^{r\times r}$, the change
$(U,S,V)\mapsto(UG_U,G_U^\top SG_V,VG_V)$ leaves $USV^\top$ unchanged.
The factor equations and \eqref{eq:factor-transfers} are equivariant
under this change; hence the final matrix is independent of factor
gauges and QR choices.
Every equation in \eqref{eq:five-substeps} is solved exactly.  We call a
step \emph{rank-admissible} if every K or L factor that must be
QR-factorized has rank $r$ and every transferred core is nonsingular.
Under this condition, all QR factorizations and core transfers in
\eqref{eq:factor-transfers} are defined.

Denote the resulting map by $\Psi_{t,h}$ and define
\begin{equation}\label{eq:numerical-method}
 Y^{n+1}:=\Psi_{t_n,h}(Y^n),\qquad t_n:=nh.
\end{equation}
The proposed robust second-order statement would assert the existence
of constants $C>0$ and $h_0>0$, depending only on
$m,n,r,B_F,L_F,T,\kappa$, such that, for every problem satisfying
\eqref{eq:basic-bounds}--\eqref{eq:smooth-bounds} and every stepsize
$0<h\le h_0$ for which the exact-subflow run is rank-admissible, the
numerical solution satisfies
\begin{equation}\label{eq:target-intro}
 \max_{t_n\le T}\norm{Y^n-A(t_n)}_{\F}
 \le C(\delta+\eps+h^2),
\end{equation}
where the maximum is over the grid points $t_n=nh\le T$.  Neither $C$
nor $h_0$ may depend on a retained singular value.  The family in
Section~\ref{sec:construction}
has $h_j\to0$ and therefore eventually satisfies $h_j\le h_0$ for
every fixed $h_0>0$.
We prove that, along these admissible stepsizes, the ratio of the
left-hand side of \eqref{eq:target-intro} to
$\delta_j+\eps_j+h_j^2$ is unbounded.  Hence no such uniform constants
$C$ and $h_0$ exist.

\section{The counterexample family}
\label{sec:construction}

The construction combines three ingredients.  The initial matrix has a
second singular value of size $h_j^3$.  The term $R_j(t)$ rotates its
associated left and right factor directions, while the two matrix
functions $M_+$ and $M_-$ agree at the limiting matrix $E_{11}$ but have
different first derivatives.  Subtracting the two problems cancels their
common first-order term and leaves the nonzero $h_j^2$ term computed in
Section~\ref{sec:defect}.

For every $j\ge3$ we construct two equations, denoted by their
right-hand sides $F_{j,+}$ and $F_{j,-}$.  We write both equations at
once using \(\eta\in\{+1,-1\}\).  Let $E_{ab}$ denote the
$3\times3$ matrix unit and put
\begin{equation}\label{eq:hj}
 h_j:=2^{-j^2},
 \qquad A_{j,\eta}(0):=Y_{j,\eta}^0:=E_{11}+h_j^3E_{22}.
\end{equation}
Both the full equation and the numerical method use this initial value.
Thus $\delta_j=0$, and the retained singular values are $1$ and
$h_j^3$.

For $X=(x_{ab})\in\mathbb R^{3\times3}$, define
\begin{equation}\label{eq:M-field}
 \Omega:=E_{31}-E_{13},
\end{equation}
\begin{equation}\label{eq:q-field}
 q_\eta(X):=\frac{1+\eta\tanh(x_{31})}
 {\sqrt{1+\norm{X}_{\F}^2}},
 \qquad M_\eta(X):=q_\eta(X)(\Omega X-X\Omega).
\end{equation}
The two matrix functions have the same value at the rank-one limit of
the initial matrices, whereas their first derivatives differ:
\begin{align}
 M_+(E_{11})=M_-(E_{11})
 &=\frac{E_{31}+E_{13}}{\sqrt2},\notag\\
 D(M_+-M_-)(E_{11})[Z]
 &=\sqrt2\,z_{31}(E_{31}+E_{13}),
 \qquad Z=(z_{ab}).
 \label{eq:pair-design}
\end{align}
Next define
\begin{equation}\label{eq:nu-W}
 \boldsymbol{\nu}(\theta):=
 \begin{pmatrix}0\\ \cos(\pi\theta/4)\\ \sin(\pi\theta/4)\end{pmatrix},
 \qquad W(\theta):=\boldsymbol{\nu}(\theta)
 \boldsymbol{\nu}(\theta)^\top.
\end{equation}
Thus $\boldsymbol\nu(\theta)$ is a unit vector rotating in the
$(\boldsymbol e_2,\boldsymbol e_3)$-plane and $W(\theta)$ is the
orthogonal projector onto its span; in particular, $W(\theta+4)=W(\theta)$.
To reduce the angular speed smoothly from one to zero, set
\begin{equation}\label{eq:psi-def}
 \psi(x):=x+15x^4-39x^5+34x^6-10x^7.
\end{equation}
Define
\begin{equation}\label{eq:beta}
 \beta:=\max_{0\le x\le1}|\psi'(x)|=\frac{27904}{16807}
\end{equation}
and
\begin{equation}\label{eq:phase-def}
 \vartheta_j(t):=
 \begin{cases}
  t/h_j,&0\le t\le4jh_j,\\
  4j+4\psi\bigl((t-4jh_j)/(4h_j)\bigr),
       &4jh_j\le t\le4(j+1)h_j,\\
  4(j+1),&t\ge4(j+1)h_j.
 \end{cases}
\end{equation}
Thus $\vartheta_j(t)=t/h_j$ on $0\le t\le4jh_j$, its speed decreases
smoothly to zero on $4jh_j\le t\le4(j+1)h_j$, and it is constant
thereafter.  Since $4(j+1)2^{-j^2}<1$ for $j\ge3$, all three pieces
occur inside the interval $[0,1]$.
On $[0,1]$, define
\begin{equation}\label{eq:R-def}
 R_j(t):=h_j^3\vartheta_j'(t)W'(\vartheta_j(t))
 =\frac{d}{dt}\bigl(h_j^3W(\vartheta_j(t))\bigr).
\end{equation}
Since $W(\theta+4)=W(\theta)$, this gives
\begin{equation}\label{eq:periodic-forcing}
 R_j(t)=h_j^2W'(t/h_j),\qquad 0\le t\le4jh_j,
\end{equation}
followed by the four-step switch-off interval
$[4jh_j,4(j+1)h_j]$ and then $R_j(t)=0$.

The two right-hand sides are
\begin{equation}\label{eq:full-field}
 F_{j,\eta}(t,X):=M_\eta(X)+R_j(t),
 \qquad \eta\in\{+1,-1\}.
\end{equation}
The endpoint identities verified in Proposition~\ref{prop:family-bounds}
imply
$R_j\in C^2([0,1];\R^{3\times3})$.

Write $Y_{j,\eta}^n$ for the iterates of \eqref{eq:numerical-method}
with right-hand side $F_{j,\eta}$ and stepsize $h_j$.

For the $n$th Strang step, its \emph{discrete phase index} is
$\ell:=n\bmod4\in\{0,1,2,3\}$.  A \emph{four-step block} is the ordered
composition of the four phase indices $0,1,2,3$ and therefore contains twenty
factor substeps.  On the first block, write
\begin{equation}\label{eq:four-step-map}
 \Psi^{[4]}_{h_j,\eta}:=
 \Psi^{j,\eta}_{3h_j,h_j}\circ
  \Psi^{j,\eta}_{2h_j,h_j}\circ
  \Psi^{j,\eta}_{h_j,h_j}\circ
  \Psi^{j,\eta}_{0,h_j}.
\end{equation}
Here $\Psi^{j,\eta}_{t,h_j}$ denotes
\eqref{eq:numerical-method} with the right-hand side $F_{j,\eta}$.
Let $\Phi_{j,\eta}(t;t_0,X_0)$ denote the solution value at time $t$ with
$A(t_0)=X_0$.  For the problem indexed by $(j,\eta)$, define
\begin{equation}\label{eq:block-objects}
 t_{k,j}:=4kh_j,
 \qquad 0\le k\le j.
\end{equation}
\begin{equation}
 \tau_{k,j,\eta}
 :=Y_{j,\eta}^{4(k+1)}
 -\Phi_{j,\eta}(t_{k+1,j};t_{k,j},Y_{j,\eta}^{4k}),
 \qquad 0\le k<j.
 \label{eq:block-local-error}
\end{equation}
For $k=0$, \eqref{eq:block-local-error} reads
\begin{equation}\label{eq:first-block-local-error}
 \tau_{0,j,\eta}
 =\Psi^{[4]}_{h_j,\eta}(Y^0_{j,\eta})
 -\Phi_{j,\eta}(4h_j;0,Y^0_{j,\eta}).
\end{equation}

\begin{theorem}[Main result]
\label{thm:main}
For the $3\times3$, rank-two family above, let $h_j=2^{-j^2}$.  For every $j\ge3$ and
$\eta\in\{+1,-1\}$, the equation $\dot A=F_{j,\eta}(t,A)$ is defined
on $[0,1]$ and satisfies, for all $t\in[0,1]$ and
$X,Z\in\mathbb R^{3\times3}$,
\begin{equation}\label{eq:theorem-basic-bounds}
 \norm{F_{j,\eta}(t,X)}_{\F}\le5,
 \qquad
 \norm{F_{j,\eta}(t,X)-F_{j,\eta}(t,Z)}_{\F}
 \le12\norm{X-Z}_{\F}.
\end{equation}
For every $Y\in\mathcal M_2$,
\begin{align}
 F_{j,\eta}(t,Y)&=M_\eta(Y)+R_j(t),\notag\\
 M_\eta(Y)&\in T_Y\mathcal M_2,\qquad
 \norm{M_\eta(Y)}_{\F}\le4,\notag\\
 \norm{R_j(t)}_{\F}&\le\eps_j,
 \label{eq:theorem-normal-bound}
\end{align}
where
\begin{equation}\label{eq:theorem-epsilon}
 \eps_j:=\beta\sqrt2\,\frac\pi4h_j^2,
 \qquad
 \norm{Y^0_{j,\eta}-A_{j,\eta}(0)}_{\F}=0.
\end{equation}
The two retained singular values of $Y^0_{j,\eta}$ are $1$ and
$h_j^3$.
Moreover,
\begin{equation}\label{eq:theorem-smooth-bounds}
 \max\!\left\{
 \norm{D_X^2F_{j,\eta}}_{(\F,\F)\to\F},
 \norm{\partial_tF_{j,\eta}}_{\F},
 \norm{\partial_t^2F_{j,\eta}}_{\F},
 \norm{D_X\partial_tF_{j,\eta}}_{\F\to\F}
 \right\}\le123.
\end{equation}
For all sufficiently large $j$ and both values $\eta=+1,-1$, every
factor substep in every complete grid step whose endpoint does not
exceed $1$ is well defined, and
\begin{equation}\label{eq:main-lower}
 \max_{\eta\in\{+1,-1\}}
 \norm{Y_{j,\eta}^{4j}-A_{j,\eta}(4jh_j)}_{\F}
 \ge \frac{\sqrt2}{64}j h_j^2.
\end{equation}
Consequently, \eqref{eq:target-intro} cannot hold with $C$ and the
admissible stepsize bound depending only on the fixed problem-class
constants and not on the retained singular values.  Thus the classical
Strang projector-splitting integrator does not attain robust
second-order convergence under the stated assumptions.
\end{theorem}

\begin{proposition}[Uniform bounds]
\label{prop:family-bounds}
For every $j\ge3$ and each $\eta\in\{+1,-1\}$, $F_{j,\eta}$ satisfies
\eqref{eq:basic-bounds}--\eqref{eq:smooth-bounds} with
\begin{equation}\label{eq:fixed-constants}
 B_F=5,\qquad L_F=12,\qquad \kappa=123,
 \qquad \eps_j=\beta\sqrt2\,\frac{\pi}{4}h_j^2,
 \qquad \delta_j=0.
\end{equation}
\end{proposition}

\begin{proof}
For a rank-two matrix $Y$, the curve
\(s\mapsto e^{s\Omega}Ye^{-s\Omega}\) remains rank two and has derivative
$\Omega Y-Y\Omega$ at $s=0$.  Hence
$M_\eta(Y)\in T_Y\mathcal M_2$.
Moreover,
\[
 \norm{W'}_{\F}=\sqrt2\,\frac\pi4,
 \qquad |\vartheta_j'|\le\frac\beta {h_j},
\]
and therefore \(\norm{R_j(t)}_{\F}\le\eps_j\).

Write $r:=\norm X_{\F}$, $s:=(1+r^2)^{1/2}$,
$a_\eta(X):=1+\eta\tanh(x_{31})$, and $g(X):=s^{-1}$, so that
$q_\eta=a_\eta g$.  Since $|a_\eta|\le2$,
$\norm{Da_\eta}\le1$, $\norm{D^2a_\eta}<1$,
$\norm{Dg}\le r/s^3$, and
$\norm{D^2g}\le(1+4r^2)/s^5$, the product rule gives
\begin{align*}
 \norm{Dq_\eta}&\le\frac1s+\frac{2r}{s^3},&
 \norm{D^2q_\eta}&\le
 \frac1s+\frac{2r}{s^3}+\frac{2(1+4r^2)}{s^5}.
\end{align*}
Together with $|q_\eta(X)|\le2s^{-1}$ and
$\norm{\Omega X-X\Omega}_{\F}\le2r$, this yields
\begin{align*}
 \norm{M_\eta(X)}_{\F}&\le4,&
 \norm{DM_\eta(X)}_{\F\to\F}&\le10,\\
 \norm{D^2M_\eta(X)}_{(\F,\F)\to\F}
 &\le2r\norm{D^2q_\eta}+4\norm{Dq_\eta}\le34<36,
\end{align*}
where $r/s\le1$, $r/s^3\le1$, and $(1+4r^2)/s^4\le4$ were used.
The polynomial in \eqref{eq:psi-def} satisfies
\[
 \psi'(x)=(1-x)^3(70x^3+6x^2+3x+1)\ge0,
 \qquad 0\le x\le1,
\]
$\psi''(x)=-60x^2(x-1)^2(7x-3)$.  In addition,
$\psi(0)=0$ and $\psi(1)=1$.
The maximum of $\psi'$ occurs at $x=3/7$ and equals the value $\beta$ in
\eqref{eq:beta}.  Moreover,
$\psi'(0)=1$, $\psi'(1)=0$, and
$\psi''(0)=\psi''(1)=\psi'''(0)=\psi'''(1)=0$.  These identities make the three
pieces of $R_j$ twice continuously differentiable.  On the switch-off
interval, put $x:=(t-4jh_j)/(4h_j)$.  Direct differentiation gives
\begin{align}
 R_j'(t)
 &=h_j\left[
 \frac14\psi''(x)W'(4\psi(x))
 +\bigl(\psi'(x)\bigr)^2W''(4\psi(x))\right],
 \label{eq:R-first-derivative}\\
 R_j''(t)
 &=\frac1{16}\psi'''(x)W'(4\psi(x))
 +\frac34\psi'(x)\psi''(x)W''(4\psi(x))
 +\bigl(\psi'(x)\bigr)^3W'''(4\psi(x)).
 \label{eq:R-second-derivative}
\end{align}
Moreover,
\[
 \norm{W'}_{\F}=\sqrt2\frac\pi4,\quad
 \norm{W''}_{\F}=2\sqrt2\left(\frac\pi4\right)^2,\quad
 \norm{W'''}_{\F}=4\sqrt2\left(\frac\pi4\right)^3,
\]
while $\max|\psi''|\le15$, $\max|\psi'''|\le1110$, and
$\max|\psi'|=\beta$.  Hence
\begin{align*}
 \norm{R_j'(t)}_{\F}
 &\le h_j\left[
 \frac{15}{4}\sqrt2\frac\pi4
 +\beta^2 2\sqrt2\left(\frac\pi4\right)^2\right]
 <9h_j,\\
 \norm{R_j''(t)}_{\F}
 &\le \frac{1110}{16}\sqrt2\frac\pi4
 +\frac{45\beta}{4}2\sqrt2\left(\frac\pi4\right)^2
 +\beta^3 4\sqrt2\left(\frac\pi4\right)^3
 <123,\\
 D_X\partial_tF_{j,\eta}&=0.
\end{align*}
Before the switch-off interval,
$R_j'=h_jW''(t/h_j)$ and $R_j''=W'''(t/h_j)$; afterwards both
derivatives vanish.  The endpoint identities above make these formulas
agree at the two joins.
Since $\norm{R_j}_{\F}<1$ for $j\ge3$, the constants in
\eqref{eq:fixed-constants} follow.
\end{proof}

\section{The error in the first four-step block}
\label{sec:defect}

All $O(\cdot)$ and $o(\cdot)$ estimates in this section are uniform in
$\ell\in\{0,1,2,3\}$, $i\in\{1,\ldots,5\}$, and
$\eta\in\{+1,-1\}$.

Along the sequence $h_j\to0$, equation
\eqref{eq:first-block-local-error} and the identity
$Y^0_{j,+}=Y^0_{j,-}$ give
\begin{equation}\label{eq:tau-difference-decomposition}
\begin{aligned}
 \tau_{0,j,+}-\tau_{0,j,-}
 ={}&\Bigl(\Psi^{[4]}_{h_j,+}(Y_{j,+}^0)
           -\Psi^{[4]}_{h_j,-}(Y_{j,-}^0)\Bigr)\\
 &-\Bigl(\Phi_{j,+}(4h_j;0,Y_{j,+}^0)
          -\Phi_{j,-}(4h_j;0,Y_{j,-}^0)\Bigr).
\end{aligned}
\end{equation}
Set
\begin{equation}\label{eq:H-def}
 H:=\Omega E_{11}-E_{11}\Omega=E_{31}+E_{13}.
\end{equation}

\begin{lemma}[Ranks and projection matrices]
\label{lem:rank-projectors}
For step $\ell\in\{0,1,2,3\}$ of the first four-step block, let
$P^{h_j}_{\ell,\eta}$ be the row-space projector at the beginning of
the step, $Q^{h_j}_{\ell,\eta}$ the column-space projector after
$K_1$, and $P^{h_j}_{\ell+1,\eta}$ the row-space projector after $L$.
For all sufficiently large $j$, every K and L factor occurring in
these four steps has column rank two throughout its substep, and every
S factor is nonsingular.  Moreover,
\begin{align}
 P^{h_j}_{\ell,\eta}
 &=E_{11}+W(\ell)+O(h_j),\notag\\
 Q^{h_j}_{\ell,\eta}
 &=E_{11}+W(\ell+\tfrac12)+O(h_j),\notag\\
 P^{h_j}_{\ell+1,\eta}
 &=E_{11}+W(\ell+1)+O(h_j),
 \label{eq:projector-limits}
\end{align}
\end{lemma}

\begin{proof}
Subsection~\ref{app:positive-rank-margins} computes the rank margins for
the $h^0$ factor cycle.  Lemma~\ref{lem:desingularized-rank-neighborhood}
transfers them to $h=h_j>0$, and
Lemma~\ref{lem:desingularized-block-map} then gives the three projector
expansions.

For
$K=[\boldsymbol k_1,\boldsymbol k_2]$,
$L=[\boldsymbol\lambda_1,\boldsymbol\lambda_2]$, and
$S=(s_{iq})_{i,q=1}^2$, the required rank conditions are
\begin{equation}\label{eq:rank-conditions}
 \boldsymbol k_1\times\boldsymbol k_2\ne0,
 \qquad \det S\ne0,
 \qquad \boldsymbol\lambda_1\times\boldsymbol\lambda_2\ne0.
\end{equation}
Subsection~\ref{app:positive-rank-margins} proves fixed positive lower
bounds for the three desingularized rank quantities
$h_j^{-3}(\boldsymbol k_1\times\boldsymbol k_2)$,
$h_j^{-3}\det S$, and
$h_j^{-3}(\boldsymbol\lambda_1\times\boldsymbol\lambda_2)$.
Lemma~\ref{lem:desingularized-rank-neighborhood} transfers
these bounds to the factors at the actual stepsize $h_j$.  Therefore
\eqref{eq:rank-conditions} holds throughout all five substeps,
uniformly for all sufficiently large $j$.  Hence
$\rank K=\rank L=2$ and $\det S\ne0$.

At the three interfaces that determine the projection matrices,
Lemma~\ref{lem:desingularized-block-map} gives
\[
 V_0=[\boldsymbol e_1,\boldsymbol\nu(\ell)]+O(h_j),\quad
 U_1=[\boldsymbol e_1,\boldsymbol\nu(\ell+\tfrac12)]+O(h_j),\quad
 V_1=[\boldsymbol e_1,\boldsymbol\nu(\ell+1)]+O(h_j).
\]
Multiplication by the corresponding transpose gives
\begin{align*}
 P^{h_j}_{\ell,\eta}
 &=\bigl([\boldsymbol e_1,\boldsymbol\nu(\ell)]+O(h_j)\bigr)
   \bigl([\boldsymbol e_1,\boldsymbol\nu(\ell)]+O(h_j)\bigr)^\top
   =E_{11}+W(\ell)+O(h_j),\\
 Q^{h_j}_{\ell,\eta}
 &=\bigl([\boldsymbol e_1,\boldsymbol\nu(\ell+\tfrac12)]+O(h_j)\bigr)
   \bigl([\boldsymbol e_1,\boldsymbol\nu(\ell+\tfrac12)]+O(h_j)\bigr)^\top
   =E_{11}+W(\ell+\tfrac12)+O(h_j),\\
 P^{h_j}_{\ell+1,\eta}
 &=\bigl([\boldsymbol e_1,\boldsymbol\nu(\ell+1)]+O(h_j)\bigr)
   \bigl([\boldsymbol e_1,\boldsymbol\nu(\ell+1)]+O(h_j)\bigr)^\top
   =E_{11}+W(\ell+1)+O(h_j).
\end{align*}
This proves \eqref{eq:projector-limits}.
\end{proof}

\begin{lemma}[Difference of the four-step numerical solutions]
\label{lem:numerical-difference}
As $j\to\infty$,
\begin{equation}\label{eq:numerical-difference}
 \Psi^{[4]}_{h_j,+}(Y_{j,+}^0)-\Psi^{[4]}_{h_j,-}(Y_{j,-}^0)
 =h_j^2\left(8H+\frac{\sqrt2}{16}E_{12}\right)+O(h_j^3).
\end{equation}
\end{lemma}

\begin{proof}
We first compare the projector expansions for $\eta=+1$ and $\eta=-1$.
We then expand each of the twenty exact substeps through order $h_j^2$.
Finally, we sum the four phase contributions.

For step $\ell\in\{0,1,2,3\}$, let
$X^{h_j}_{\ell,\eta,0}$ be its input and let
$X^{h_j}_{\ell,\eta,i}$ be the output after substep $i$, where
\begin{equation}\label{eq:twenty-state-chain}
 X^{h_j}_{0,\eta,0}:=Y_{j,\eta}^0,
 \qquad X^{h_j}_{\ell+1,\eta,0}:=X^{h_j}_{\ell,\eta,5},
 \qquad
 \Psi^{[4]}_{h_j,\eta}(Y_{j,\eta}^0)=X^{h_j}_{3,\eta,5}.
\end{equation}
Using the three projectors from Lemma~\ref{lem:rank-projectors}, define
the corresponding five operators $\mathcal P^{h_j}_{\ell,\eta,i}$ by
\begin{equation}\label{eq:actual-five-actions}
\begin{aligned}
 \mathcal P^{h_j}_{\ell,\eta,1}(Z)&:=ZP^{h_j}_{\ell,\eta},
 &\mathcal P^{h_j}_{\ell,\eta,2}(Z)&:=-Q^{h_j}_{\ell,\eta}ZP^{h_j}_{\ell,\eta},
 &\mathcal P^{h_j}_{\ell,\eta,3}(Z)&:=Q^{h_j}_{\ell,\eta}Z,\\
 \mathcal P^{h_j}_{\ell,\eta,4}(Z)&:=-Q^{h_j}_{\ell,\eta}ZP^{h_j}_{\ell+1,\eta},
 &\mathcal P^{h_j}_{\ell,\eta,5}(Z)&:=ZP^{h_j}_{\ell+1,\eta}.
\end{aligned}
\end{equation}
At $E_{11}$, the two right-hand sides satisfy
\begin{equation}\label{eq:base-field-difference}
 F_{j,+}(t,E_{11})-F_{j,-}(t,E_{11})
 =M_+(E_{11})-M_-(E_{11})
 =\frac1{\sqrt2}H-\frac1{\sqrt2}H=0.
\end{equation}
Lemma~\ref{lem:desingularized-block-map} gives Taylor expansions of the
three projectors in the desingularized variables at stepsize $h_j$.
Every occurrence of $q_\eta$ in those equations is multiplied by $h_j$,
and
$\left.\partial_h(hq_\eta(X(h)))\right|_{h=0}
=q_\eta(E_{11})=1/\sqrt2$ for $\eta\in\{+1,-1\}$.
Thus $Dq_\eta$ first enters at order $h_j^2$, and the constant and linear
projector coefficients are the same for both values of $\eta$; see
\eqref{eq:appendix-equal-first-jets}--
\eqref{eq:appendix-projector-derivative-equality}.  Consequently,
\begin{align*}
 P^{h_j}_{\ell,\eta}
 &=E_{11}+W(\ell)+h_jP^{(1)}_{\ell}+O(h_j^2),\\
 Q^{h_j}_{\ell,\eta}
 &=E_{11}+W(\ell+\tfrac12)+h_jQ^{(1)}_{\ell}+O(h_j^2),\\
 P^{h_j}_{\ell+1,\eta}
 &=E_{11}+W(\ell+1)+h_jP^{(1)}_{\ell+1}+O(h_j^2),
\end{align*}
where \(P^{(1)}_{\ell},Q^{(1)}_{\ell},P^{(1)}_{\ell+1}\) do not
depend on \(\eta\).  Subtraction therefore gives
\begin{equation}\label{eq:projector-sign-difference}
 P^{h_j}_{\ell,+}-P^{h_j}_{\ell,-}=O(h_j^2),\qquad
 Q^{h_j}_{\ell,+}-Q^{h_j}_{\ell,-}=O(h_j^2),\qquad
 P^{h_j}_{\ell+1,+}-P^{h_j}_{\ell+1,-}=O(h_j^2).
\end{equation}
For example, the second identity in
\eqref{eq:actual-five-actions} and
\eqref{eq:projector-sign-difference} give
\[
 \mathcal P^{h_j}_{\ell,+,2}(Z)-\mathcal P^{h_j}_{\ell,-,2}(Z)
 =-(Q^{h_j}_{\ell,+}-Q^{h_j}_{\ell,-})
 ZP^{h_j}_{\ell,+}
 -Q^{h_j}_{\ell,-}Z
 (P^{h_j}_{\ell,+}-P^{h_j}_{\ell,-})
 =O(h_j^2\norm Z_{\F}).
\]
The other four identities in \eqref{eq:actual-five-actions} give
\begin{equation}\label{eq:action-sign-difference}
 \max_{1\le i\le5}
 \norm{\mathcal P^{h_j}_{\ell,+,i}
       -\mathcal P^{h_j}_{\ell,-,i}}_{\F\to\F}
 =O(h_j^2).
\end{equation}
Substitution of \eqref{eq:projector-limits} into these five formulas
gives directly
\begin{equation}\label{eq:five-actions}
\begin{aligned}
 \mathcal P^{h_j}_{\ell,\eta,1}(Z)
 &=Z\bigl(E_{11}+W(\ell)\bigr)+O(h_j\norm Z_{\F}),\\
 \mathcal P^{h_j}_{\ell,\eta,2}(Z)
 &=-\bigl(E_{11}+W(\ell+\tfrac12)\bigr)
 Z\bigl(E_{11}+W(\ell)\bigr)+O(h_j\norm Z_{\F}),\\
 \mathcal P^{h_j}_{\ell,\eta,3}(Z)
 &=\bigl(E_{11}+W(\ell+\tfrac12)\bigr)Z
 +O(h_j\norm Z_{\F}),\\
 \mathcal P^{h_j}_{\ell,\eta,4}(Z)
 &=-\bigl(E_{11}+W(\ell+\tfrac12)\bigr)
 Z\bigl(E_{11}+W(\ell+1)\bigr)+O(h_j\norm Z_{\F}),\\
 \mathcal P^{h_j}_{\ell,\eta,5}(Z)
 &=Z\bigl(E_{11}+W(\ell+1)\bigr)+O(h_j\norm Z_{\F}).
\end{aligned}
\end{equation}
For $i=1,\ldots,5$, denote the displayed matrix product before the
$O(h_j\norm Z_{\F})$ term by $\mathcal P_{\ell,i}(Z)$ and set
\begin{equation}\label{eq:substep-lengths}
 (\alpha_1,\ldots,\alpha_5):=(1/2,1/2,1,1/2,1/2).
\end{equation}
Let $I^{h_j}_{\ell,i}$ be the time interval of substep $i$ in
\eqref{eq:five-substeps}, shifted to the step
$[\ell h_j,(\ell+1)h_j]$.  If $X^{h_j}_{\ell,\eta,i}(s)$ denotes its solution
inside that interval, then the definition of the five exact substeps gives
\begin{align}
 &\Psi^{[4]}_{h_j,+}(Y_{j,+}^0)-\Psi^{[4]}_{h_j,-}(Y_{j,-}^0)\notag\\
 &\quad=\sum_{\ell=0}^3\sum_{i=1}^5
 \int_{I^{h_j}_{\ell,i}}
 \Bigl[
 \mathcal P^{h_j}_{\ell,+,i}
 F_{j,+}\bigl(s,X^{h_j}_{\ell,+,i}(s)\bigr)
 -\mathcal P^{h_j}_{\ell,-,i}
 F_{j,-}\bigl(s,X^{h_j}_{\ell,-,i}(s)\bigr)
 \Bigr]\,ds.
 \label{eq:numerical-integral-definition}
\end{align}
For every $(\ell,i)$, the part of the integrand in
\eqref{eq:numerical-integral-definition} caused by the difference of
the two actions satisfies
\begin{align}
 &\norm{\int_{I^{h_j}_{\ell,i}}
 \bigl(\mathcal P^{h_j}_{\ell,+,i}
       -\mathcal P^{h_j}_{\ell,-,i}\bigr)
 F_{j,+}\bigl(s,X^{h_j}_{\ell,+,i}(s)\bigr)\,ds}_{\F}
 \notag\\
 &\qquad\le |I^{h_j}_{\ell,i}|\,
 O(h_j^2)\sup_{s,X}\norm{F_{j,+}(s,X)}_{\F}
 =O(h_j^3).
 \label{eq:projector-difference-order}
\end{align}
On the first four steps, $\norm{R_j(t)}_{\F}=O(h_j^2)$ and
$\norm{\partial_tR_j(t)}_{\F}=O(h_j)$.
Hence the direct contribution of $R_j$ over a substep of length
$O(h_j)$, as well as every time-dependent Taylor term, is $O(h_j^3)$.
Taylor expansion of each integral in
\eqref{eq:numerical-integral-definition} gives, for $i=1,\ldots,5$,
\begin{align}
 X^{h_j}_{\ell,\eta,i}-X^{h_j}_{\ell,\eta,i-1}
 ={}&\alpha_i h_j\mathcal P^{h_j}_{\ell,\eta,i}
      \bigl(M_\eta(X^{h_j}_{\ell,\eta,i-1})\bigr)\notag\\
 &+\frac{\alpha_i^2h_j^2}{2}
 \mathcal P^{h_j}_{\ell,\eta,i}\!\left(
 DM_\eta(X^{h_j}_{\ell,\eta,i-1})\left[
 \mathcal P^{h_j}_{\ell,\eta,i}
 \bigl(M_\eta(X^{h_j}_{\ell,\eta,i-1})\bigr)
 \right]\right)
 +O(h_j^3).
 \label{eq:one-substep-taylor}
\end{align}
The first estimate below follows from the uniform field bound over the
twenty substeps.  For the second, suppose at one substep input that
$X_\eta=E_{11}+O(h_j)$ and $X_+-X_-=O(h_j^2)$.
Since $M_+-M_-$ vanishes at $E_{11}$ and is Lipschitz near that matrix,
$\norm{M_+(X_+)-M_-(X_-)}_{\F}
\le C\norm{X_+-X_-}_{\F}+C\norm{X_--E_{11}}_{\F}=O(h_j)$.
Together with \eqref{eq:action-sign-difference}, multiplication by the
substep length gives an $O(h_j^2)$ output difference.  Starting from
$Y^0_{j,+}=Y^0_{j,-}$ and repeating this implication through the finite
chain \eqref{eq:twenty-state-chain} proves
\begin{equation}\label{eq:actual-limit-estimates}
 X^{h_j}_{\ell,\eta,i}=E_{11}+O(h_j),\qquad
 X^{h_j}_{\ell,+,i}-X^{h_j}_{\ell,-,i}=O(h_j^2).
\end{equation}
Moreover, direct substitution of the five operators in
\eqref{eq:five-actions} gives
\begin{equation}\label{eq:first-order-identity}
\begin{aligned}
 \sum_{i=1}^5\alpha_i\mathcal P_{\ell,i}(H/\sqrt2)
 &=\frac1{2\sqrt2}\bigl(
 H(E_{11}+W(\ell))\\
 &\qquad -(E_{11}+W(\ell+\tfrac12))H(E_{11}+W(\ell))\\
 &\qquad +2(E_{11}+W(\ell+\tfrac12))H\\
 &\qquad -(E_{11}+W(\ell+\tfrac12))H(E_{11}+W(\ell+1))\\
 &\qquad +H(E_{11}+W(\ell+1))\bigr)\\
 &=\frac1{\sqrt2}H.
\end{aligned}
\end{equation}
For $Z=(z_{ab})\in\mathbb R^{3\times3}$, direct differentiation of
\eqref{eq:q-field} gives
\begin{equation}\label{eq:centered-derivative}
 M_+(E_{11})=M_-(E_{11})=\frac1{\sqrt2}H,
 \qquad
 D(M_+-M_-)(E_{11})[Z]=\sqrt2\,z_{31}H.
\end{equation}

Substituting \eqref{eq:centered-derivative},
\eqref{eq:projector-limits}, and \eqref{eq:actual-limit-estimates} into
\eqref{eq:one-substep-taylor}, and summing according to
\eqref{eq:twenty-state-chain}, first gives
\begin{equation}\label{eq:first-order-state}
 X^{h_j}_{\ell,\eta,i}
 =E_{11}+h_j\left(
 \frac{\ell}{\sqrt2}H
 +\sum_{p\le i}\alpha_p\mathcal P_{\ell,p}(H/\sqrt2)
 \right)+O(h_j^2).
\end{equation}
Substituting \eqref{eq:first-order-state} back into
\eqref{eq:one-substep-taylor}, subtracting the two values of $\eta$,
and summing the twenty resulting equalities gives
\begin{align}
 &\frac{\Psi^{[4]}_{h_j,+}(Y_{j,+}^0)-\Psi^{[4]}_{h_j,-}(Y_{j,-}^0)}{h_j^2}
 \notag\\
 &\quad=
 \sum_{\ell=0}^3\sum_{i=1}^5
 \alpha_i\mathcal P_{\ell,i}\!\left(
 D(M_+-M_-)(E_{11})\!\left[
 \frac{\ell}{\sqrt2}H
 +
   \sum_{p<i}\alpha_p\mathcal P_{\ell,p}(H/\sqrt2)
 \right]\right)
 \notag\\
 &\qquad\quad+\frac12
 \sum_{\ell=0}^3\sum_{i=1}^5\alpha_i^2
 \mathcal P_{\ell,i}\!\left(
 D(M_+-M_-)(E_{11})
 [\mathcal P_{\ell,i}(H/\sqrt2)]\right)+O(h_j).
 \label{eq:numerical-difference-direct}
\end{align}
The terms containing $\ell H/\sqrt2$ equal
\begin{equation}\label{eq:previous-step-sum}
 \sum_{\ell=0}^3\ell
 \sum_{i=1}^5\alpha_i\mathcal P_{\ell,i}\!\left(
 D(M_+-M_-)(E_{11})[H/\sqrt2]\right)
 =\sum_{\ell=0}^3\ell H=6H.
\end{equation}
The remaining terms in \eqref{eq:numerical-difference-direct} are
explicit $3\times3$ matrix products involving only $W(\ell)$,
$W(\ell+1/2)$, and $W(\ell+1)$.  For example,
\begin{align*}
 W(0)&=E_{22},\\
 W(1/2)&=\cos^2(\pi/8)E_{22}
 +\cos(\pi/8)\sin(\pi/8)(E_{23}+E_{32})\\
 &\quad+\sin^2(\pi/8)E_{33},\\
 W(1)&=\tfrac12(E_{22}+E_{23}+E_{32}+E_{33}).
\end{align*}
Substitution into \eqref{eq:numerical-difference-direct} gives, for
$\ell=0$,
\[
 \frac12H+\left(\frac{\sqrt2}{64}+\frac1{32}\right)E_{12}
 +\left(\frac{\sqrt2}{64}+\frac1{32}\right)E_{13}.
\]
The same substitution for $\ell=1,2,3$ gives the other rows below.
\begin{equation}\label{eq:within-step-table}
\begin{array}{c|l}
\ell&\text{contribution after removing the history term in
\eqref{eq:previous-step-sum}}\\
\hline
0&\displaystyle \frac12H+(\sqrt2/64+1/32)E_{12}
 +(\sqrt2/64+1/32)E_{13}\\[1mm]
1&\displaystyle \frac12H+(-1/32+\sqrt2/64)E_{12}
 +(1/32-\sqrt2/64)E_{13}\\[1mm]
2&\displaystyle \frac12H+(-1/32+\sqrt2/64)E_{12}
 +(-1/32+\sqrt2/64)E_{13}\\[1mm]
3&\displaystyle \frac12H+(\sqrt2/64+1/32)E_{12}
 +(-1/32-\sqrt2/64)E_{13}\\[1mm]
\hline
\displaystyle\sum_{\ell=0}^3
 &\displaystyle 2H+\frac{\sqrt2}{16}E_{12}.
\end{array}
\end{equation}
The $E_{13}$ coefficients cancel, whereas the four $E_{12}$
coefficients sum to $\sqrt2/16$.  Hence
\begin{align*}
 \frac{\Psi^{[4]}_{h_j,+}(Y_{j,+}^0)-\Psi^{[4]}_{h_j,-}(Y_{j,-}^0)}{h_j^2}
 &=6H+2H+\frac{\sqrt2}{16}E_{12}+O(h_j)\\
 &=8H+\frac{\sqrt2}{16}E_{12}+O(h_j),
\end{align*}
which is \eqref{eq:numerical-difference}.
\end{proof}

\begin{lemma}[Difference of the full solutions]
\label{lem:full-flow-difference}
As $j\to\infty$,
\begin{equation}\label{eq:full-flow-difference}
 \Phi_{j,+}(4h_j;0,Y_{j,+}^0)-\Phi_{j,-}(4h_j;0,Y_{j,-}^0)
 =8h_j^2H+O(h_j^3).
\end{equation}
\end{lemma}

\begin{proof}
On the first four steps,
\begin{equation}\label{eq:full-flow-small-terms}
 Y_{j,\eta}^0=E_{11}+h_j^3W(0),
 \qquad
 \int_0^tR_j(s)\,ds=h_j^3\bigl(W(t/h_j)-W(0)\bigr),
 \qquad 0\le t\le4h_j,
\end{equation}
and hence
\begin{equation}\label{eq:full-flow-forcing-partial-bound}
 \sup_{0\le t\le4h_j}
 \norm{\int_0^tR_j(s)\,ds}_{\F}=O(h_j^3).
\end{equation}
Moreover,
\begin{equation}\label{eq:full-flow-forcing-cycle}
 \int_0^{4h_j}R_j(s)\,ds=h_j^3\bigl(W(4)-W(0)\bigr)=0.
\end{equation}
The direct forcing contribution is therefore $O(h_j^3)$, while its
interaction with $M_\eta$ is $O(h_j^4)$.  Taylor expansion gives
\begin{align}
 \Phi_{j,\eta}(4h_j;0,Y_{j,\eta}^0)
 ={}&E_{11}+4h_jM_\eta(E_{11})
 +8h_j^2DM_\eta(E_{11})[M_\eta(E_{11})]+O(h_j^3),
 \label{eq:full-flow-expansion}\\
 &\Phi_{j,+}(4h_j;0,Y_{j,+}^0)
 -\Phi_{j,-}(4h_j;0,Y_{j,-}^0)\notag\\
 &\qquad=8h_j^2D(M_+-M_-)(E_{11})[H/\sqrt2]
 +O(h_j^3)\notag\\
 &\qquad=8h_j^2H+O(h_j^3),
\end{align}
where the last equality follows from \eqref{eq:centered-derivative}.
\end{proof}

\begin{proposition}[Error in the first four-step block]
\label{prop:four-step}
As $j\to\infty$, equivalently $h_j\to0^+$,
\begin{equation}\label{eq:four-step-defect}
 \tau_{0,j,+}-\tau_{0,j,-}
 =\frac{\sqrt2}{16}h_j^2E_{12}+O(h_j^3).
\end{equation}
\end{proposition}

\begin{proof}
Substitute \eqref{eq:numerical-difference} and
\eqref{eq:full-flow-difference} into
\eqref{eq:tau-difference-decomposition}.
\end{proof}

\section{From one four-step error to the global error}
\label{sec:accumulation}

The times $t_{k,j}$ and four-step errors $\tau_{k,j,\eta}$ are those in
\eqref{eq:block-objects}--\eqref{eq:block-local-error}.
Proposition~\ref{prop:four-step} gives the error difference in the
first block:
\begin{equation}\label{eq:first-block-coefficient-section-five}
 \ip{\tau_{0,j,+}-\tau_{0,j,-}}{E_{12}}
 =\frac{\sqrt2}{16}h_j^2+O(h_j^3).
\end{equation}
We first extend this formula from $k=0$ to every $0\le k<j$.  Write
$t=t_{k,j}+s$ on $[t_{k,j},t_{k+1,j}]$.  Since
$t_{k,j}=4kh_j$ and $W'(\theta+4)=W'(\theta)$, for
$0\le s\le4h_j$,
\begin{equation}\label{eq:block-periodicity}
\begin{aligned}
 R_j(t_{k,j}+s)&=h_j^2W'\left(4k+\frac{s}{h_j}\right)
 =h_j^2W'\left(\frac{s}{h_j}\right)=R_j(s),\\
 F_{j,\eta}(t_{k,j}+s,X)
 &=M_\eta(X)+R_j(t_{k,j}+s)=F_{j,\eta}(s,X).
\end{aligned}
\end{equation}
Hence the equation on $[t_{k,j},t_{k+1,j}]$, written in the variable
$s=t-t_{k,j}$, is independent of $k$; only its initial matrix
$Y_{j,\eta}^{4k}$ changes.  The following proposition shows that these
matrices remain close to $Y^0_{j,\eta}$ and that the
coefficient computed in Section~\ref{sec:defect} therefore persists in
every block.

\begin{proposition}[Uniform local error over the periodic blocks]
\label{prop:uniform-block}
For all sufficiently large $j$, uniformly for $0\le k<j$,
\begin{align}
 \norm{\tau_{k,j,\eta}}_{\F}&\le Ch_j^2,
 \label{eq:block-bound}\\
 \ip{\tau_{k,j,+}-\tau_{k,j,-}}{E_{12}}
 &=\frac{\sqrt2}{16}h_j^2+o(h_j^2).
 \label{eq:block-centered}
\end{align}
\end{proposition}

\begin{proof}
We first show that the numerical and full four-step maps agree through
order $h_j$ for every block start near the reference factors.  Uniform
$C^3$ dependence then gives an $h_j^2$ local error whose coefficient is
Lipschitz in the normalized block-start data.  A first-exit argument
keeps all $j$ periodic block starts in that neighborhood.

Fix $0\le k<j$, set
\[
 t=t_{k,j}+h_j\theta,
 \qquad 0\le\theta\le4,
\]
and write the full equation in this variable.  The chain rule,
\eqref{eq:block-periodicity}, and \eqref{eq:periodic-forcing} give
\begin{align}
 \frac{dX}{d\theta}
 &=h_j\dot X(t_{k,j}+h_j\theta)\notag\\
 &=h_jF_{j,\eta}(t_{k,j}+h_j\theta,X)\notag\\
 &=h_j\bigl(M_\eta(X)+R_j(h_j\theta)\bigr)\notag\\
 &=h_jM_\eta(X)+h_j^3W'(\theta).
 \label{eq:actual-block-rescaling}
\end{align}
Substituting the right-hand side of
\eqref{eq:actual-block-rescaling} into the five equations in
\eqref{eq:five-substeps} gives the K, S, and L equations on fixed
$\theta$-intervals.

At the beginning of block $k$, write
\begin{equation}\label{eq:block-factorization}
 Y_{j,\eta}^{4k}=U_{k,j,\eta}S_{k,j,\eta}V_{k,j,\eta}^{\top},
 \qquad
 S_{k,j,\eta}=(s_{iq}^{k,j,\eta})_{i,q=1}^2,
\end{equation}
where $U_{k,j,\eta}$ and $V_{k,j,\eta}$ have orthonormal columns.
At every QR factorization, choose each new column to have positive
inner product with the reference direction
$\boldsymbol\nu(\ell+\tfrac12)$ after $K_1$ and
$\boldsymbol\nu(\ell+1)$ after $L$ and $K_2$.  After each block
apply the fixed gauge
\begin{equation}\label{eq:block-gauge-main}
 (U,S,V)\longmapsto(UG,G^\top SG,VG),
 \qquad G:=\begin{pmatrix}1&0\\0&-1\end{pmatrix}.
\end{equation}
This transformation leaves both $USV^\top$ and all subsequent factor
subproblems unchanged; see Subsection~\ref{app:four-step-gauge}.
Define
\begin{equation}\label{eq:block-core-determinant}
 \Delta_{k,j,\eta}:=h_j^{-3}\det S_{k,j,\eta}
\end{equation}
and measure the departure from the initial factors by
\begin{align}
 d_{k,j,\eta}:={}&
 \norm{U_{k,j,\eta}-[\boldsymbol e_1,\boldsymbol e_2]}_{\F}
 +\norm{V_{k,j,\eta}-[\boldsymbol e_1,\boldsymbol e_2]}_{\F}
 \notag\\
 &+|s_{11}^{k,j,\eta}-1|+|s_{12}^{k,j,\eta}|
 +|s_{21}^{k,j,\eta}|+|\Delta_{k,j,\eta}-1|.
 \label{eq:block-factor-distance}
\end{align}
The initial factorization gives
\begin{equation}\label{eq:block-factor-distance-initial}
 d_{0,j,\eta}=0.
\end{equation}
Lemma~\ref{lem:desingularized-rank-neighborhood} supplies a fixed number
\(0<\rho<1/2\).  Whenever \(d_{k,j,\eta}\le\rho\),
\begin{equation}\label{eq:block-core-chart}
 |s_{11}^{k,j,\eta}|\ge1-\rho>\frac12,
 \qquad
 s_{22}^{k,j,\eta}
 =\frac{s_{12}^{k,j,\eta}s_{21}^{k,j,\eta}
       +h_j^3\Delta_{k,j,\eta}}{s_{11}^{k,j,\eta}}.
\end{equation}

We next estimate the four-step error.  For
\(d_{k,j,\eta}\le\rho\), \eqref{eq:block-core-chart} gives
\begin{align}
 S_{k,j,\eta}
 &=\begin{pmatrix}
 s_{11}^{k,j,\eta}&s_{12}^{k,j,\eta}\\
 s_{21}^{k,j,\eta}&
 \dfrac{s_{12}^{k,j,\eta}s_{21}^{k,j,\eta}
       +h_j^3\Delta_{k,j,\eta}}{s_{11}^{k,j,\eta}}
 \end{pmatrix}\notag\\
 &=\begin{pmatrix}s_{11}^{k,j,\eta}\\s_{21}^{k,j,\eta}\end{pmatrix}
   \begin{pmatrix}1&s_{12}^{k,j,\eta}/s_{11}^{k,j,\eta}\end{pmatrix}
 +h_j^3\begin{pmatrix}
 0&0\\0&\Delta_{k,j,\eta}/s_{11}^{k,j,\eta}
 \end{pmatrix},
 \label{eq:block-core-positive-step-expansion}\\
 \overline Y_{j,\eta}^{\,4k}
 &:=U_{k,j,\eta}
 \begin{pmatrix}s_{11}^{k,j,\eta}\\s_{21}^{k,j,\eta}\end{pmatrix}
 \left(
 V_{k,j,\eta}
 \begin{pmatrix}1\\s_{12}^{k,j,\eta}/s_{11}^{k,j,\eta}\end{pmatrix}
 \right)^\top
 =:\boldsymbol u\boldsymbol v^\top,
 \label{eq:block-leading-matrix}\\
 Y_{j,\eta}^{4k}
 &=\overline Y_{j,\eta}^{\,4k}
 +h_j^3U_{k,j,\eta}
 \begin{pmatrix}
 0&0\\0&\Delta_{k,j,\eta}/s_{11}^{k,j,\eta}
 \end{pmatrix}
 V_{k,j,\eta}^\top.
 \label{eq:block-leading-relation}
\end{align}
Thus \(\overline Y_{j,\eta}^{\,4k}\) is the rank-one limit of
\eqref{eq:block-leading-relation} in the desingularized coordinates.
Consider one of
the four Strang steps in the $h=0$ desingularized factor path supplied by
Lemma~\ref{lem:desingularized-block-map}.  Let
$(U_0,S_0,V_0)$ be its input factors, let $U_1$ be the basis obtained
after $K_1$, and let $V_1$ be the basis obtained after $L$.  As in
\eqref{eq:projectors-one-step}, define
\begin{equation}\label{eq:block-proof-projectors}
 P_0:=V_0V_0^\top,\qquad Q:=U_1U_1^\top,\qquad
 P_1:=V_1V_1^\top.
\end{equation}
The column spaces of $U_0,U_1$ contain $\boldsymbol u$, and the column
spaces of $V_0,V_1$, which represent the corresponding matrix row
spaces, contain $\boldsymbol v$.  Hence
\[
 Q\boldsymbol u=\boldsymbol u,\qquad
 P_0\boldsymbol v=P_1\boldsymbol v=\boldsymbol v.
\]
Using the substep lengths in \eqref{eq:substep-lengths} gives the
first-order sum
\begin{align}
 &\frac12M_\eta(\overline Y_{j,\eta}^{\,4k})P_0
 -\frac12QM_\eta(\overline Y_{j,\eta}^{\,4k})P_0
 +QM_\eta(\overline Y_{j,\eta}^{\,4k})\notag\\
 &\qquad
 -\frac12QM_\eta(\overline Y_{j,\eta}^{\,4k})P_1
 +\frac12M_\eta(\overline Y_{j,\eta}^{\,4k})P_1\notag\\
 &\quad=QM_\eta(\overline Y_{j,\eta}^{\,4k})
 +(I-Q)M_\eta(\overline Y_{j,\eta}^{\,4k})\frac{P_0+P_1}{2}\notag\\
 &\quad=M_\eta(\overline Y_{j,\eta}^{\,4k}).
 \label{eq:block-first-order-identity}
\end{align}
For the last equality, use
$M_\eta(\overline Y_{j,\eta}^{\,4k})
=q_\eta(\overline Y_{j,\eta}^{\,4k})
(\Omega\overline Y_{j,\eta}^{\,4k}
-\overline Y_{j,\eta}^{\,4k}\Omega)$ and
\[
 (I-Q)(-\overline Y_{j,\eta}^{\,4k}\Omega)=0,
 \qquad
 ((I-Q)\Omega\overline Y_{j,\eta}^{\,4k})P_0
 =((I-Q)\Omega\overline Y_{j,\eta}^{\,4k})P_1
 =(I-Q)\Omega\overline Y_{j,\eta}^{\,4k}.
\]
Equations \eqref{eq:block-local-error},
\eqref{eq:actual-block-rescaling}, and
\eqref{eq:block-first-order-identity} give
\begin{align}
 \tau_{k,j,\eta}
 &=Y_{j,\eta}^{4(k+1)}
 -\Phi_{j,\eta}(t_{k+1,j};t_{k,j},Y_{j,\eta}^{4k})\notag\\
 &=\left(
 \overline Y_{j,\eta}^{\,4k}
 +h_j\sum_{\ell=0}^3M_\eta(\overline Y_{j,\eta}^{\,4k})
 +O(h_j^2)\right)
 -\left(
 \overline Y_{j,\eta}^{\,4k}
 +h_j\int_0^4M_\eta(\overline Y_{j,\eta}^{\,4k})\,d\theta
 +O(h_j^2)\right)\notag\\
 &=\left(\overline Y_{j,\eta}^{\,4k}
 +4h_jM_\eta(\overline Y_{j,\eta}^{\,4k})+O(h_j^2)\right)-\left(
 \overline Y_{j,\eta}^{\,4k}
 +4h_jM_\eta(\overline Y_{j,\eta}^{\,4k})+O(h_j^2)\right)\notag\\
 &=O(h_j^2)
 \label{eq:block-error-first-order-cancellation}
\end{align}
whenever \(d_{k,j,\eta}\le\rho\).

At $h=0$, the factor cycle
\eqref{eq:four-step-factor-chain}, followed by the gauge
\eqref{eq:block-gauge-main}, maps the reference tuple
\[
 ([\boldsymbol e_1,\boldsymbol e_2],
  [\boldsymbol e_1,\boldsymbol e_2],1,0,0,1)
\]
to itself.  The uniform $C^1$ bound in
Lemma~\ref{lem:desingularized-block-map} therefore gives
\begin{equation}\label{eq:block-factor-distance-recurrence}
 d_{k+1,j,\eta}
 \le C_0d_{k,j,\eta}+C_1h_j
\end{equation}
whenever $d_{k,j,\eta}\le\rho$.
The uniform $C^3$ bound and
\eqref{eq:block-error-first-order-cancellation} give, on the same
neighborhood,
\begin{align}
 \tau_{k,j,\eta}
 &=h_j^2\mathcal C_\eta
   (U_{k,j,\eta},V_{k,j,\eta},s_{11}^{k,j,\eta},
    s_{12}^{k,j,\eta},s_{21}^{k,j,\eta},\Delta_{k,j,\eta})
   +O(h_j^3),
 \label{eq:block-error-taylor}
\end{align}
where $\mathcal C_\eta$ is one half of the second derivative at $h=0$
of the difference between the four-step numerical endpoint and the
corresponding full-flow endpoint, evaluated at the displayed factor
tuple.  On the fixed neighborhood,
\begin{equation}\label{eq:block-error-taylor-bounds}
 \norm{\mathcal C_\eta}_{\F}
 +\|D\mathcal C_\eta\|\le C,
\end{equation}
and the remainder in \eqref{eq:block-error-taylor} is bounded by
$Ch_j^3$.
Here $D\mathcal C_\eta$ denotes the derivative with respect to the six
argument blocks---the two Stiefel factors and four scalars---using the
product of their Euclidean and Frobenius norms.  Since the arguments for
$k=0$ equal the reference
tuple displayed above,
equations \eqref{eq:block-error-taylor}--
\eqref{eq:block-error-taylor-bounds} imply
\begin{align}
 \norm{\tau_{k,j,\eta}}_{\F}
 &\le Ch_j^2,
 \label{eq:block-error-direct-bound}\\
 \norm{\tau_{k,j,\eta}-\tau_{0,j,\eta}}_{\F}
 &\le Ch_j^2d_{k,j,\eta}+Ch_j^3.
 \label{eq:block-error-start-dependence}
\end{align}

We next show that $d_{k,j,\eta}\le\rho$ throughout the first
$j$ blocks.  If $k_*\le j$ were the first index
with $d_{k_*,j,\eta}>\rho$, then
\eqref{eq:block-factor-distance-recurrence} would be valid up to
$k_*-1$ and, by \eqref{eq:block-factor-distance-initial},
\begin{equation}\label{eq:block-factor-distance-convergence}
 d_{k_*,j,\eta}
 \le C_1h_j(1+C_0+\cdots+C_0^{k_*-1})
 \le C_1jh_j\max\{1,C_0\}^{j}.
\end{equation}
By \eqref{eq:hj}, for every fixed $c\ge1$,
\[
 jh_jc^j=j2^{-j^2}c^j\longrightarrow0.
\]
For all sufficiently large $j$, the last expression is smaller than
$\rho$, a contradiction.  Therefore no first exit exists and
\begin{equation}\label{eq:block-factor-distance-uniform}
 \max_{0\le k\le j}d_{k,j,\eta}
 \le C_1jh_j\max\{1,C_0\}^{j}\longrightarrow0.
\end{equation}
Consequently, Proposition~\ref{prop:four-step} and
\eqref{eq:block-factor-distance-uniform} give, uniformly for
$0\le k<j$,
\begin{align}
 \ip{\tau_{k,j,+}-\tau_{k,j,-}}{E_{12}}
 &={}
 \ip{\tau_{0,j,+}-\tau_{0,j,-}}{E_{12}}
 +O\!\left(h_j^2(d_{k,j,+}+d_{k,j,-})+h_j^3\right)\notag\\
 &=\frac{\sqrt2}{16}h_j^2+o(h_j^2).
 \label{eq:block-error-coefficient-propagation}
\end{align}
Equations \eqref{eq:block-error-direct-bound} and
\eqref{eq:block-error-coefficient-propagation} are
\eqref{eq:block-bound} and \eqref{eq:block-centered}, respectively.
\end{proof}

\begin{proposition}[Rank-admissibility of the entire numerical trajectory]
\label{prop:full-rank-admissibility}
For all sufficiently large $j$, the exact K--S--L--S--K trajectory is
rank-admissible on $[0,1]$.
\end{proposition}

\begin{proof}
The rank margins in Lemma~\ref{lem:desingularized-rank-neighborhood}, together
with \eqref{eq:block-factor-distance-uniform}, show that the first $j$
periodic blocks are rank-admissible for all sufficiently large $j$.
The switch-off margins in
Subsection~\ref{app:switch-off-margins}, together with
$d_{j,j,\eta}\to0$ and
Lemma~\ref{lem:desingularized-rank-neighborhood}, give the same
conclusion for the next four steps.  Thereafter $R_j(t)=0$, and
Lemma~\ref{lem:post-switch-rank} shows that every K, S, and L subproblem
preserves its input rank.  Thus the complete trajectory is
rank-admissible on $[0,1]$.
\end{proof}

We now transport each four-step local error through the remaining full
flow and telescope the resulting differences.  Equation
\eqref{eq:block-local-error} gives
\[
 Y_{j,\eta}^{4(k+1)}
 =\Phi_{j,\eta}(t_{k+1,j};t_{k,j},Y_{j,\eta}^{4k})
 +\tau_{k,j,\eta},
\]
while the composition identity for the full flow is
\[
 \Phi_{j,\eta}(t_{j,j};t_{k,j},Y_{j,\eta}^{4k})
 =\Phi_{j,\eta}\left(
 t_{j,j};t_{k+1,j},
 \Phi_{j,\eta}(t_{k+1,j};t_{k,j},Y_{j,\eta}^{4k})\right).
\]
Consequently, writing $D_X\Phi_{j,\eta}(t;s,X)$ for the derivative of
the full flow with respect to its initial matrix $X$,
\begin{align}
 e_{j,\eta}
 &:=Y_{j,\eta}^{4j}-A_{j,\eta}(4jh_j)\notag\\
 &=Y_{j,\eta}^{4j}
   -\Phi_{j,\eta}(t_{j,j};t_{0,j},Y_{j,\eta}^{0})\notag\\
 &=\sum_{k=0}^{j-1}\left[
 \Phi_{j,\eta}(t_{j,j};t_{k+1,j},Y_{j,\eta}^{4(k+1)})
 -\Phi_{j,\eta}(t_{j,j};t_{k,j},Y_{j,\eta}^{4k})\right]\notag\\
 &=\sum_{k=0}^{j-1}\left[
 \Phi_{j,\eta}(t_{j,j};t_{k+1,j},Y_{j,\eta}^{4(k+1)})
 -\Phi_{j,\eta}\left(
 t_{j,j};t_{k+1,j},
 \Phi_{j,\eta}(t_{k+1,j};t_{k,j},Y_{j,\eta}^{4k})\right)\right]
 \notag\\
 &=\sum_{k=0}^{j-1}\underbrace{\left[\int_0^1
 D_X\Phi_{j,\eta}
 \left(t_{j,j};t_{k+1,j},
 \Phi_{j,\eta}(t_{k+1,j};t_{k,j},Y_{j,\eta}^{4k})
 +\lambda\tau_{k,j,\eta}\right)\,d\lambda\right]}_{=:J_{k,j,\eta}}
 \tau_{k,j,\eta}\notag\\
 &=\sum_{k=0}^{j-1}J_{k,j,\eta}\tau_{k,j,\eta}.
 \label{eq:telescope}
\end{align}
The variational equation gives, uniformly in $k$ and $\eta$,
\begin{equation}\label{eq:transport-close}
 \norm{J_{k,j,\eta}-I}_{\F\to\F}
 \le e^{4L_Fjh_j}-1=o(1).
\end{equation}
Subtracting \eqref{eq:telescope} for $\eta=+1$ and $\eta=-1$ gives
\begin{align}
 &\ip{e_{j,+}-e_{j,-}}{E_{12}}\notag\\
 &\quad=\sum_{k=0}^{j-1}
 \ip{\tau_{k,j,+}-\tau_{k,j,-}}{E_{12}}
 +\sum_{k=0}^{j-1}
 \left\langle
 \begin{aligned}
 &(J_{k,j,+}-I)\tau_{k,j,+}\\[-1mm]
 &\quad -(J_{k,j,-}-I)\tau_{k,j,-}
 \end{aligned},E_{12}\right\rangle_{\F}.
 \label{eq:accumulation-decomposition}
\end{align}

\begin{proof}[Proof of Theorem~\ref{thm:main}]
Let $j$ be sufficiently large.  By
Proposition~\ref{prop:full-rank-admissibility}, the numerical trajectory
is rank-admissible on $[0,1]$.  Proposition~\ref{prop:uniform-block}
gives
\begin{equation}\label{eq:sum-block-errors}
 \sum_{k=0}^{j-1}
 \ip{\tau_{k,j,+}-\tau_{k,j,-}}{E_{12}}
 =\frac{\sqrt2}{16}jh_j^2+o(jh_j^2).
\end{equation}
Moreover, \eqref{eq:block-bound} and \eqref{eq:transport-close} give
\begin{align}
 &\left|\sum_{k=0}^{j-1}
 \ip{(J_{k,j,+}-I)\tau_{k,j,+}
 -(J_{k,j,-}-I)\tau_{k,j,-}}{E_{12}}\right|\notag\\
 &\quad\le
 \max_{\substack{0\le k<j\\ \eta\in\{+1,-1\}}}
 \norm{J_{k,j,\eta}-I}_{\F\to\F}
 \sum_{k=0}^{j-1}
 \left(\norm{\tau_{k,j,+}}_{\F}+\norm{\tau_{k,j,-}}_{\F}\right)
 =o(jh_j^2).
 \label{eq:sum-transport-errors}
\end{align}
Substituting \eqref{eq:sum-block-errors} and
\eqref{eq:sum-transport-errors} into
\eqref{eq:accumulation-decomposition} yields
\begin{equation}\label{eq:accumulated}
 \ip{e_{j,+}-e_{j,-}}{E_{12}}
 =\frac{\sqrt2}{16}jh_j^2+o(jh_j^2).
\end{equation}
For all sufficiently large $j$, the absolute value of the left-hand
side is at least $(\sqrt2/32)jh_j^2$.  Since $\norm{E_{12}}_{\F}=1$,
\begin{equation*}
 \frac{\sqrt2}{32}jh_j^2
 \le\bigl|\ip{e_{j,+}-e_{j,-}}{E_{12}}\bigr|
 \le\norm{e_{j,+}}_{\F}+\norm{e_{j,-}}_{\F}.
\end{equation*}
Therefore
\begin{equation}\label{eq:two-problem-lower-bound}
 \max_{\eta\in\{+1,-1\}}\norm{e_{j,\eta}}_{\F}
 =\max_{\eta\in\{+1,-1\}}
 \norm{Y_{j,\eta}^{4j}-A_{j,\eta}(4jh_j)}_{\F}
 \ge\frac{\sqrt2}{64}jh_j^2.
\end{equation}
Finally,
\[
 \delta_j+\eps_j+h_j^2
 =\left(1+\beta\sqrt2\,\frac\pi4\right)h_j^2.
\]
Since $4jh_j\le1$, the left-hand side of
\eqref{eq:target-intro}, maximized over the two problems, is at least
the quantity in \eqref{eq:two-problem-lower-bound}.  Therefore
\begin{equation*}
 \frac{
  \displaystyle\max_{\eta\in\{+1,-1\}}\max_{t_n\le1}
  \norm{Y_{j,\eta}^n-A_{j,\eta}(t_n)}_{\F}}
 {\delta_j+\eps_j+h_j^2}
 \ge
 \frac{\sqrt2}{64(1+\beta\sqrt2\,\pi/4)}j
 \longrightarrow\infty,
\end{equation*}
which proves the theorem.
\end{proof}

\appendix
\section{The desingularized factor cycle and rank preservation}
\label{app:uniformity}

Throughout this appendix, $h\ge0$ is an independent parameter; the
periodic blocks of the family in Section~\ref{sec:construction} are
obtained by setting $h=h_j$.  Recall
$\eta\in\{+1,-1\}$,
$\Omega:=E_{31}-E_{13}$ and
$x_{31}:=\boldsymbol e_3^\top X\boldsymbol e_1$, together with
\[
 q_\eta(X):=\frac{1+\eta\tanh(x_{31})}
 {\sqrt{1+\norm{X}_{\F}^2}},
 \qquad M_\eta(X):=q_\eta(X)(\Omega X-X\Omega).
\]
For $h>0$ on a periodic block, set
\[
 R_h(t):=h^2W'(t/h),\qquad F_{h,\eta}(t,X):=M_\eta(X)+R_h(t).
\]
This appendix proves, from the factor equations for $F_{h,\eta}$, that
\[
 \rank K=\rank L=2,
 \qquad \det S\ne0
\]
throughout every substep.  These are the rank statements used in
Lemma~\ref{lem:rank-projectors} and
Proposition~\ref{prop:full-rank-admissibility}.  Subsection~\ref{app:desingularized-factor-coordinates}
removes the scale $h^3$ from the rank-determining quantities;
Subsections~\ref{app:leading-factor-cycle}--\ref{app:four-step-gauge}
compute the coefficient of $h^0$, its positive margins, and
the four-step gauge; Subsections~\ref{app:uniform-neighborhood}--
\ref{app:uniform-block-expansions} transfer these results to $h>0$.
\begin{equation}\label{eq:factor-definitions}
\begin{aligned}
 \boldsymbol\nu(\theta)&:=
 \begin{pmatrix}0\\ \cos(\pi\theta/4)\\ \sin(\pi\theta/4)\end{pmatrix},
 &W(\theta)&:=\boldsymbol\nu(\theta)\boldsymbol\nu(\theta)^\top,\\
 \mathcal B(\theta)&:=[\boldsymbol e_1,\boldsymbol\nu(\theta)],
 &\Sigma_h(\zeta)&:=\begin{pmatrix}1&0\\0&h^3\zeta\end{pmatrix},
 &\gamma&:=\cos(\pi/8).
\end{aligned}
\end{equation}
Every $\mathcal B(\theta)\in\mathbb R^{3\times2}$ has orthonormal columns.  The
identities
\begin{equation}\label{eq:nu-inner-product}
\begin{aligned}
 W(\theta_1)\boldsymbol e_1&=0,\\
 \boldsymbol\nu(\theta_1)^\top\boldsymbol\nu(\theta_2)
 &=\cos\!\left(\frac{\pi(\theta_1-\theta_2)}4\right),\\
 W(\theta_1)\boldsymbol\nu(\theta_2)
 &=\cos\!\left(\frac{\pi(\theta_1-\theta_2)}4\right)
 \boldsymbol\nu(\theta_1).
\end{aligned}
\end{equation}
will be used in each substep.  The
initial factorization is
\begin{equation}\label{eq:factor-start}
 (U,S,V)=\left(
 [\boldsymbol e_1,\boldsymbol e_2],
 \begin{pmatrix}1&0\\0&h^3\end{pmatrix},
 [\boldsymbol e_1,\boldsymbol e_2]\right)
 =(\mathcal B(0),\Sigma_h(1),\mathcal B(0)).
\end{equation}

\subsection{Desingularized factor coordinates}
\label{app:desingularized-factor-coordinates}

The second retained singular value has size \(h^3\).  Accordingly,
the rank indicators in the three factor subproblems have the scale
\[
 \boldsymbol k_1\times\boldsymbol k_2=O(h^3),
 \qquad
 \det S=O(h^3),
 \qquad
 \boldsymbol\lambda_1\times\boldsymbol\lambda_2=O(h^3).
\]
Direct differentiation of the corresponding QR maps can therefore
introduce \(h^{-3}\).  We instead use
\[
 h^{-3}(\boldsymbol k_1\times\boldsymbol k_2),
 \qquad
 h^{-3}\det S,
 \qquad
 h^{-3}(\boldsymbol\lambda_1\times\boldsymbol\lambda_2),
\]
and retain the additional scalar components needed to reconstruct the
full factors.  These variables desingularize the original K, S, and L
equations.  After the cancellations displayed below, the resulting
coordinate equations and reconstruction formulas contain no negative
power of \(h\); this is the uniformity used in
Subsections~\ref{app:uniform-neighborhood} and
\ref{app:uniform-block-expansions}.

Fix $\ell\in\{0,1,2,3\}$, set $t=h\theta$ for
$\ell\le\theta\le\ell+1$, and write
$K=[\boldsymbol k_1,\boldsymbol k_2]$,
$L=[\boldsymbol\lambda_1,\boldsymbol\lambda_2]$, and
$S=(s_{iq})_{i,q=1}^2$.  For a matrix $Z$ with two columns,
$[Z]_i$ denotes its $i$th column.  By definition,
$R_h(h\theta)=h^2W'(\theta)$.
The K equation associated with \eqref{eq:five-substeps} therefore gives
\begin{equation}\label{eq:appendix-K-equation}
 \frac{dK}{d\theta}
 =hF_{h,\eta}(h\theta,KV^\top)V
 =hM_\eta(KV^\top)V+h^3W'(\theta)V.
\end{equation}
Since $V^\top V=I$,
\begin{equation}\label{eq:appendix-K-columns}
 M_\eta(KV^\top)V
 =q_\eta(KV^\top)\bigl(\Omega K-K(V^\top\Omega V)\bigr),
\end{equation}
and hence, for $i=1,2$,
\begin{equation}\label{eq:appendix-K-column-equation}
 \frac{d\boldsymbol k_i}{d\theta}
 =hq_\eta(KV^\top)
  \bigl(\Omega\boldsymbol k_i-[K(V^\top\Omega V)]_i\bigr)
  +h^3[W'(\theta)V]_i.
\end{equation}
Define the scaled column area
\begin{equation}\label{eq:appendix-K-area-definition}
 \boldsymbol w_K:=h^{-3}(\boldsymbol k_1\times\boldsymbol k_2).
\end{equation}
For $A\in\mathbb R^{3\times3}$ and
$\boldsymbol a,\boldsymbol b\in\mathbb R^3$,
\begin{equation}\label{eq:cross-product-linear-identity}
 (A\boldsymbol a)\times\boldsymbol b
 +\boldsymbol a\times(A\boldsymbol b)
 =\bigl(\operatorname{tr}(A)I-A^\top\bigr)
  (\boldsymbol a\times\boldsymbol b).
\end{equation}
Thus, for the skew-symmetric matrix $\Omega$, the right-hand side is
$\Omega(\boldsymbol a\times\boldsymbol b)$.  For a two-by-two matrix
$C$, one likewise has
$[KC]_1\times\boldsymbol k_2+\boldsymbol k_1\times[KC]_2
=\operatorname{tr}(C)(\boldsymbol k_1\times\boldsymbol k_2)$.
Differentiating this definition and substituting
\eqref{eq:appendix-K-column-equation} gives
\begin{equation}\label{eq:appendix-normalized-K}
\begin{aligned}
 \frac{d\boldsymbol w_K}{d\theta}
={}&h^{-3}\left(
 \frac{d\boldsymbol k_1}{d\theta}\times\boldsymbol k_2
 +\boldsymbol k_1\times\frac{d\boldsymbol k_2}{d\theta}
 \right)\\
={}&h^{-2}q_\eta(KV^\top)\Bigl(
 \Omega(\boldsymbol k_1\times\boldsymbol k_2)
 -\operatorname{tr}(V^\top\Omega V)
  (\boldsymbol k_1\times\boldsymbol k_2)\Bigr)\\
 &+[W'(\theta)V]_1\times\boldsymbol k_2
 +\boldsymbol k_1\times[W'(\theta)V]_2\\
={}&hq_\eta(KV^\top)\Omega\boldsymbol w_K
 +[W'(\theta)V]_1\times\boldsymbol k_2
 +\boldsymbol k_1\times[W'(\theta)V]_2.
\end{aligned}
\end{equation}
The second equality uses
\eqref{eq:cross-product-linear-identity} and
$(V^\top\Omega V)^\top=-V^\top\Omega V$, hence
$\operatorname{tr}(V^\top\Omega V)=0$.
Together with \(\boldsymbol k_1\), the scaled area does not yet
determine \(\boldsymbol k_2\); its component parallel to
\(\boldsymbol k_1\) is also needed.  Define
\begin{equation}\label{eq:appendix-K-inner-product}
 p_K:=\boldsymbol k_1^\top\boldsymbol k_2.
\end{equation}
Differentiating \(p_K\), using \(\Omega^\top=-\Omega\) and
\((V^\top\Omega V)^\top=-V^\top\Omega V\), gives
\begin{align}
 \frac{dp_K}{d\theta}
 ={}&hq_\eta(KV^\top)
 \bigl((V^\top\Omega V)K^\top K
       -K^\top K(V^\top\Omega V)\bigr)_{12}\notag\\
 &+h^3\left(
 [W'(\theta)V]_1^\top\boldsymbol k_2
 +\boldsymbol k_1^\top[W'(\theta)V]_2\right).
 \label{eq:appendix-normalized-K-parallel}
\end{align}

During an S substep, $U$ and $V$ are fixed.  Using
\(t=h\theta\) in the S equation in
\eqref{eq:five-substeps} gives
\begin{align}
 \frac{dS}{d\theta}
 ={}&-hq_\eta(USV^\top)
 \bigl((U^\top\Omega U)S-S(V^\top\Omega V)\bigr)
 -h^3U^\top W'(\theta)V.
 \label{eq:appendix-S-equation}
\end{align}
Define the scaled core determinant
\begin{equation}\label{eq:appendix-core-determinant-definition}
 \Delta_S:=h^{-3}\det S.
\end{equation}
Jacobi's formula and \eqref{eq:appendix-S-equation} give
\begin{align}
 \frac{d\,\Delta_S}{d\theta}
 &={}h^{-3}\operatorname{tr}\left(
 \operatorname{adj}(S)\frac{dS}{d\theta}\right)\notag\\
 &={}-h^{-2}q_\eta(USV^\top)\det(S)
 \left(\operatorname{tr}(U^\top\Omega U)
 -\operatorname{tr}(V^\top\Omega V)\right)\notag\\
 &\quad-\operatorname{tr}\bigl(
 \operatorname{adj}(S)U^\top W'(\theta)V\bigr)\notag\\
 &={}-\operatorname{tr}\bigl(
 \operatorname{adj}(S)U^\top W'(\theta)V\bigr),
 \label{eq:appendix-normalized-S}
\end{align}
where the last equality uses the skew-symmetry of
$U^\top\Omega U$ and $V^\top\Omega V$.
Together with \(s_{11},s_{12},s_{21}\), it reconstructs the complete
core.  These three entries satisfy
\begin{equation}\label{eq:appendix-core-coordinate-odes}
 \frac{ds_{11}}{d\theta}=\left[\frac{dS}{d\theta}\right]_{11},
 \qquad
 \frac{ds_{12}}{d\theta}=\left[\frac{dS}{d\theta}\right]_{12},
 \qquad
 \frac{ds_{21}}{d\theta}=\left[\frac{dS}{d\theta}\right]_{21},
\end{equation}
where \(dS/d\theta\) is the right-hand side of
\eqref{eq:appendix-S-equation}; the equation for \(\Delta_S\) is
\eqref{eq:appendix-normalized-S}.

For the L substep, $X=UL^\top$ and
\begin{equation}\label{eq:appendix-L-equation}
 \frac{dL}{d\theta}
 =hM_\eta(UL^\top)^\top U+h^3W'(\theta)U.
\end{equation}
Define
\begin{equation}\label{eq:appendix-L-area-definition}
 \boldsymbol w_L:=h^{-3}(\boldsymbol\lambda_1\times\boldsymbol\lambda_2),
 \qquad p_L:=\boldsymbol\lambda_1^\top\boldsymbol\lambda_2.
\end{equation}
Applying the calculation leading to
\eqref{eq:appendix-normalized-K}--
\eqref{eq:appendix-normalized-K-parallel} with
$(K,V,\boldsymbol k_i,p_K,\boldsymbol w_K)$ replaced by
$(L,U,\boldsymbol\lambda_i,p_L,\boldsymbol w_L)$ gives
\begin{align}
 \frac{d\boldsymbol w_L}{d\theta}
 ={}&hq_\eta(UL^\top)\Omega
 \boldsymbol w_L
 \notag\\
 &+[W'(\theta)U]_1\times\boldsymbol\lambda_2
 +\boldsymbol\lambda_1\times[W'(\theta)U]_2.
 \label{eq:appendix-normalized-L}
\end{align}
\begin{align}
 \frac{dp_L}{d\theta}
 ={}&hq_\eta(UL^\top)
 \bigl((U^\top\Omega U)L^\top L
       -L^\top L(U^\top\Omega U)\bigr)_{12}\notag\\
 &+h^3\left(
 [W'(\theta)U]_1^\top\boldsymbol\lambda_2
 +\boldsymbol\lambda_1^\top[W'(\theta)U]_2\right).
 \label{eq:appendix-normalized-L-parallel}
\end{align}

Equations \eqref{eq:appendix-K-column-equation},
\eqref{eq:appendix-normalized-K},
\eqref{eq:appendix-normalized-K-parallel},
\eqref{eq:appendix-S-equation},
\eqref{eq:appendix-normalized-S},
\eqref{eq:appendix-L-equation},
\eqref{eq:appendix-normalized-L}, and
\eqref{eq:appendix-normalized-L-parallel} contain no negative power of
$h$ after the second columns and the lower-right core entry are
reconstructed by
\begin{align}
 \boldsymbol k_2
 &=\frac{p_K}{\norm{\boldsymbol k_1}_2^2}\boldsymbol k_1
 +h^3\frac{\boldsymbol w_K\times\boldsymbol k_1}
 {\norm{\boldsymbol k_1}_2^2},
 \label{eq:appendix-QR-reconstruction}\\
 s_{22}&=\frac{s_{12}s_{21}+h^3\Delta_S}{s_{11}}.
 \label{eq:appendix-core-reconstruction}
\end{align}
The L reconstruction is the same formula with
\(\boldsymbol k_i\) replaced by \(\boldsymbol\lambda_i\).

For later use, the signed QR map itself is
\begin{align}
 \boldsymbol u_1&:=\frac{\boldsymbol k_1}{\norm{\boldsymbol k_1}_2},
 &\boldsymbol u_2&:=
 \frac{\boldsymbol w_K\times\boldsymbol k_1}
      {\norm{\boldsymbol w_K\times\boldsymbol k_1}_2},
 \label{eq:appendix-signed-QR-basis}\\
 U_1&:=[\boldsymbol u_1,\boldsymbol u_2],
 &\widehat S_1&:=
 \begin{pmatrix}
 \norm{\boldsymbol k_1}_2&p_K/\norm{\boldsymbol k_1}_2\\
 0&h^3
 \dfrac{\norm{\boldsymbol w_K\times\boldsymbol k_1}_2}
 {\norm{\boldsymbol k_1}_2^2}
 \end{pmatrix}.
 \label{eq:appendix-signed-QR-core}
\end{align}
Then \(K=U_1\widehat S_1\).  The scaled determinant of the
transferred core is
\begin{equation}\label{eq:appendix-signed-QR-determinant}
 h^{-3}\det\widehat S_1
 =\frac{\norm{\boldsymbol w_K\times\boldsymbol k_1}_2}
 {\norm{\boldsymbol k_1}_2}.
\end{equation}
The L-interface uses the identical formulas with
\(\boldsymbol k_i,U_1,\widehat S_1\) replaced by
\(\boldsymbol\lambda_i,V_1,\widehat S_2^\top\).  The remaining transfers
in \eqref{eq:factor-transfers} are identities or transposes of these
two-by-two cores.

\subsection{The leading factor cycle}
\label{app:leading-factor-cycle}

Setting $h=0$ in \eqref{eq:appendix-normalized-K},
\eqref{eq:appendix-normalized-S}, and
\eqref{eq:appendix-normalized-L} removes every $M_\eta$ term.  Hence the
$h^0$ factor equations, written in the physical time $t$, are
\begin{align*}
 \dot K_1&=R_hV_0,
 &\dot S_1&=-U_1^\top R_hV_0,
 &\dot L&=R_h^\top U_1,\\
 \dot S_2&=-U_1^\top R_hV_1,
 &\dot K_2&=R_hV_1.
\end{align*}
We integrate these equations in the displayed order.

For the step on $[\ell h,(\ell+1)h]$, where
$\ell\in\{0,1,2,3\}$, and
$\ell\le\theta_0<\theta_1\le\ell+1$,
\begin{equation}\label{eq:reference-R-integral}
 \int_{\theta_0h}^{\theta_1h}R_h(t)\,dt
 =h^2\int_{\theta_0h}^{\theta_1h}W'(t/h)\,dt
 =h^3\int_{\theta_0}^{\theta_1}W'(\theta)\,d\theta
 =h^3[W(\theta_1)-W(\theta_0)].
\end{equation}
We now apply \eqref{eq:reference-R-integral} successively to the five
factor equations in \eqref{eq:five-substeps}, using the transfers in
\eqref{eq:factor-transfers} and starting from
$(U_0,S_0,V_0)=(\mathcal B(\ell),\Sigma_h(1),\mathcal B(\ell))$.

For the first K substep,
\[
 \dot K_1(t)=R_h(t)\mathcal B(\ell),
 \qquad K_1(\ell h)=\mathcal B(\ell)\Sigma_h(1).
\]
Hence
\begin{align*}
 K_1((\ell+\tfrac12)h)
 &=\mathcal B(\ell)\Sigma_h(1)+h^3
 [W(\ell+\tfrac12)-W(\ell)]\mathcal B(\ell)\\
 &=\left[\boldsymbol e_1,
 h^3\left(\boldsymbol\nu(\ell)
 +[W(\ell+\tfrac12)-W(\ell)]\boldsymbol\nu(\ell)\right)\right]\\
 &=\mathcal B(\ell+\tfrac12)\Sigma_h(\gamma),
\end{align*}
because
\[
 \boldsymbol\nu(\ell)+
 [W(\ell+\tfrac12)-W(\ell)]\boldsymbol\nu(\ell)
 =\gamma\boldsymbol\nu(\ell+\tfrac12).
\]
Its QR factorization therefore gives
\[
 U_1=\mathcal B(\ell+\tfrac12),
 \qquad \widehat S_1=\Sigma_h(\gamma).
\]

The first S substep satisfies
\[
 \dot S_1(t)
 =-\mathcal B(\ell+\tfrac12)^\top R_h(t)\mathcal B(\ell),
 \qquad S_1(\ell h)=\Sigma_h(\gamma).
\]
Since $W(\theta)\boldsymbol e_1=0$ and
\[
 \boldsymbol\nu(\ell+\tfrac12)^\top
 [W(\ell+\tfrac12)-W(\ell)]\boldsymbol\nu(\ell)
 =\gamma-\gamma=0,
\]
equation \eqref{eq:reference-R-integral} gives
\begin{align*}
 S_1((\ell+\tfrac12)h)
 &=\Sigma_h(\gamma)-h^3\mathcal B(\ell+\tfrac12)^\top
 [W(\ell+\tfrac12)-W(\ell)]\mathcal B(\ell)\\
 &=\Sigma_h(\gamma).
\end{align*}

The L substep therefore starts from
\[
 L(\ell h)=\mathcal B(\ell)\Sigma_h(\gamma),
 \qquad \dot L(t)=R_h(t)^\top \mathcal B(\ell+\tfrac12).
\]
Since $R_h(t)^\top=R_h(t)$,
\begin{align*}
 L((\ell+1)h)
 &=\mathcal B(\ell)\Sigma_h(\gamma)+h^3
 [W(\ell+1)-W(\ell)]\mathcal B(\ell+\tfrac12)\\
 &=\left[\boldsymbol e_1,
 h^3\left(\gamma\boldsymbol\nu(\ell)
 +[W(\ell+1)-W(\ell)]\boldsymbol\nu(\ell+\tfrac12)\right)\right]\\
 &=\mathcal B(\ell+1)\Sigma_h(\gamma),
\end{align*}
where
\[
 \gamma\boldsymbol\nu(\ell)+
 [W(\ell+1)-W(\ell)]\boldsymbol\nu(\ell+\tfrac12)
 =\gamma\boldsymbol\nu(\ell+1).
\]
The QR factorization of $L((\ell+1)h)$ gives
\[
 V_1=\mathcal B(\ell+1),
 \qquad \widehat S_2=\Sigma_h(\gamma).
\]

The second S substep satisfies
\[
 \dot S_2(t)
 =-\mathcal B(\ell+\tfrac12)^\top R_h(t)\mathcal B(\ell+1),
 \qquad S_2((\ell+\tfrac12)h)=\Sigma_h(\gamma).
\]
Since
\[
 \boldsymbol\nu(\ell+\tfrac12)^\top
 [W(\ell+1)-W(\ell+\tfrac12)]\boldsymbol\nu(\ell+1)
 =\gamma-\gamma=0,
\]
we obtain
\begin{align*}
 S_2((\ell+1)h)
 &=\Sigma_h(\gamma)-h^3\mathcal B(\ell+\tfrac12)^\top
 [W(\ell+1)-W(\ell+\tfrac12)]\mathcal B(\ell+1)\\
 &=\Sigma_h(\gamma).
\end{align*}

Finally,
\[
 K_2((\ell+\tfrac12)h)=\mathcal B(\ell+\tfrac12)\Sigma_h(\gamma),
 \qquad \dot K_2(t)=R_h(t)\mathcal B(\ell+1),
\]
and therefore
\begin{align*}
 K_2((\ell+1)h)
 &=\mathcal B(\ell+\tfrac12)\Sigma_h(\gamma)+h^3
 [W(\ell+1)-W(\ell+\tfrac12)]\mathcal B(\ell+1)\\
 &=\left[\boldsymbol e_1,
 h^3\left(\gamma\boldsymbol\nu(\ell+\tfrac12)
 +[W(\ell+1)-W(\ell+\tfrac12)]\boldsymbol\nu(\ell+1)\right)\right]\\
 &=\mathcal B(\ell+1)\Sigma_h(1),
\end{align*}
because
\[
 \gamma\boldsymbol\nu(\ell+\tfrac12)+
 [W(\ell+1)-W(\ell+\tfrac12)]\boldsymbol\nu(\ell+1)
 =\boldsymbol\nu(\ell+1).
\]
The final QR factorization gives
\[
 U_2=\mathcal B(\ell+1),
 \qquad S^+=\Sigma_h(1).
\]
Thus the five derived transfers are
\begin{equation}\label{eq:one-step-factor-walk}
\begin{aligned}
 (\mathcal B(\ell),\Sigma_h(1),\mathcal B(\ell))
 &\xrightarrow{K_1}
 (\mathcal B(\ell+\tfrac12),\Sigma_h(\gamma),\mathcal B(\ell))\\
 &\xrightarrow{S_1}
 (\mathcal B(\ell+\tfrac12),\Sigma_h(\gamma),\mathcal B(\ell))\\
 &\xrightarrow{L}
 (\mathcal B(\ell+\tfrac12),\Sigma_h(\gamma),\mathcal B(\ell+1))\\
 &\xrightarrow{S_2}
 (\mathcal B(\ell+\tfrac12),\Sigma_h(\gamma),\mathcal B(\ell+1))\\
 &\xrightarrow{K_2}
 (\mathcal B(\ell+1),\Sigma_h(1),\mathcal B(\ell+1)).
\end{aligned}
\end{equation}

Substituting $\ell=0,1,2,3$ in
\eqref{eq:one-step-factor-walk} gives all twenty substeps and the four
step endpoints at order $h^0$,
\begin{equation}\label{eq:four-step-factor-chain}
\begin{aligned}
 (\mathcal B(0),\Sigma_h(1),\mathcal B(0))&\longrightarrow(\mathcal B(1),\Sigma_h(1),\mathcal B(1))
 \longrightarrow(\mathcal B(2),\Sigma_h(1),\mathcal B(2))\\
 &\longrightarrow(\mathcal B(3),\Sigma_h(1),\mathcal B(3))
 \longrightarrow(\mathcal B(4),\Sigma_h(1),\mathcal B(4)).
\end{aligned}
\end{equation}
By \eqref{eq:factor-definitions}, every $\mathcal B(\theta)$ has orthonormal columns.
Since $\gamma>0$, every displayed K and L endpoint has column rank two,
and $\det \Sigma_h(\gamma)=h^3\gamma\ne0$ and
$\det \Sigma_h(1)=h^3\ne0$.

\subsection{Positive rank margins}
\label{app:positive-rank-margins}

For $h>0$, the three rank conditions are equivalent to
\begin{equation}\label{eq:appendix-rank-equivalences}
\begin{aligned}
 \operatorname{rank}K=2
 &\iff \boldsymbol k_1\times\boldsymbol k_2\ne0
 \iff \boldsymbol w_K\ne0,\\
 \det S\ne0&\iff \Delta_S\ne0,\\
 \operatorname{rank}L=2
 &\iff \boldsymbol\lambda_1\times\boldsymbol\lambda_2\ne0
 \iff \boldsymbol w_L\ne0.
\end{aligned}
\end{equation}
Along the five equations with $M_\eta$ omitted,
\[
 \norm{\boldsymbol k_1}_2=1,
 \qquad \norm{\boldsymbol\lambda_1}_2=1,
 \qquad |s_{11}|=1.
\]
In the third column below, each symbol denotes its coefficient of
$h^0$ in the corresponding desingularized equation.  Direct
integration gives the following exact values:
{\small
\setlength{\arraycolsep}{3pt}
\begin{equation}\label{eq:appendix-rank-table}
\begin{array}{c|c|c}
\text{substep and quantity}&\text{interval}&\text{coefficient of $h^0$}\\
\hline
K_1:\ \boldsymbol w_K&\ell\le\theta\le\ell+\tfrac12&
 \cos\!\left(\frac\pi4(\theta-\ell)\right)
  \boldsymbol e_1\times\boldsymbol\nu(\theta)\\[2mm]
S_1:\ \Delta_S&\ell\le\theta\le\ell+\tfrac12&
 2\gamma-\cos\!\left(\frac\pi4(\theta-\ell)\right)
 \cos\!\left(\frac\pi8-\frac\pi4(\theta-\ell)\right)\\[2mm]
L:\ \boldsymbol w_L&\ell\le\theta\le\ell+1&
 \cos\!\left(\frac\pi4(\theta-\ell)-\frac\pi8\right)
  \boldsymbol e_1\times\boldsymbol\nu(\theta)\\[2mm]
S_2:\ \Delta_S&\ell+\tfrac12\le\theta\le\ell+1&
 2\gamma-\cos\!\left(\frac\pi4(\theta-\ell-\tfrac12)\right)
 \cos\!\left(\frac\pi8-\frac\pi4(\theta-\ell-\tfrac12)\right)\\[2mm]
K_2:\ \boldsymbol w_K&\ell+\tfrac12\le\theta\le\ell+1&
 \cos\!\left(\frac\pi4(\ell+1-\theta)\right)
  \boldsymbol e_1\times\boldsymbol\nu(\theta).
\end{array}
\end{equation}
}
Since
\[
 \norm{\boldsymbol e_1\times\boldsymbol\nu(\theta)}_2=1,
\]
the K and L rows of \eqref{eq:appendix-rank-table} satisfy
\[
 \norm{\boldsymbol w_K(\theta)}_2
 \ge\gamma,
 \qquad
 \norm{\boldsymbol w_L(\theta)}_2
 \ge\gamma.
\]
For $0\le y\le\pi/8$,
\[
 \cos y\cos(\pi/8-y)
 =\frac{\gamma+\cos(2y-\pi/8)}2
 \le\frac{\gamma+1}{2}.
\]
Hence the two S rows satisfy
\[
 \Delta_S(\theta)
 \ge 2\gamma-\frac{\gamma+1}{2}
 =\frac{3\gamma-1}{2}>0.
\]
These bounds give fixed positive margins for all three conditions in
\eqref{eq:appendix-rank-equivalences}.

\subsection{Four-step gauge normalization}
\label{app:four-step-gauge}

At the end of \eqref{eq:four-step-factor-chain},
\begin{equation}\label{eq:appendix-gauge-data}
 \boldsymbol\nu(4)=-\boldsymbol\nu(0),
 \qquad \mathcal B(4)=\mathcal B(0)G,
 \qquad G:=\begin{pmatrix}1&0\\0&-1\end{pmatrix},
 \qquad G^\top \Sigma_h(1)G=\Sigma_h(1).
\end{equation}
Consequently,
\begin{equation}\label{eq:appendix-gauge-cycle}
 (\mathcal B(4),\Sigma_h(1),\mathcal B(4))
 \longmapsto
 (\mathcal B(0),\Sigma_h(1),\mathcal B(0))
\end{equation}
under the fixed transformation
\begin{equation}\label{eq:appendix-gauge-map}
 (U,S,V)\longmapsto
 (\widetilde U,\widetilde S,\widetilde V)
 :=(UG,G^\top SG,VG).
\end{equation}
The represented matrix is unchanged:
\[
 \widetilde U\widetilde S\widetilde V^\top
 =UG(G^\top SG)G^\top V^\top=USV^\top.
\]
At a K interface, \(\widetilde K=KG\), and
\[
 F_{h,\eta}(t,\widetilde K\widetilde V^\top)\widetilde V
 =F_{h,\eta}(t,KV^\top)VG
 =\dot K G=\dot{\widetilde K}.
\]
The S and L interfaces satisfy
\[
 \dot{\widetilde S}=G^\top\dot S G,
 \qquad
 \widetilde L=LG,
 \qquad
 \dot{\widetilde L}=\dot L G.
\]
Thus \eqref{eq:appendix-gauge-map} does not change any subsequent
factor subproblem.

\subsection{Reference margins during switch-off}
\label{app:switch-off-margins}

Recall
\(
 \psi(x)=x+15x^4-39x^5+34x^6-10x^7
\)
and \(\beta=\max_{0\le x\le1}\psi'(x)\).  In the scaled variable
\(0\le\theta\le4\), the time-dependent term in the factor equations is
\begin{equation}\label{eq:switch-off-scaled}
 h^3\psi'(\theta/4)W'\bigl(4\psi(\theta/4)\bigr).
\end{equation}
For \(\ell\in\{0,1,2,3\}\), define the angular advances during the
first and second half of step \(\ell\) by
\begin{align}
 a_\ell&:=\pi\left[
 \psi\left(\frac{\ell+1/2}{4}\right)
 -\psi\left(\frac{\ell}{4}\right)\right],\notag\\
 b_\ell&:=\pi\left[
 \psi\left(\frac{\ell+1}{4}\right)
 -\psi\left(\frac{\ell+1/2}{4}\right)\right].
 \label{eq:switch-off-half-step-angles}
\end{align}
Since \(0\le\psi'\le\beta\),
\begin{equation}\label{eq:reduction-angular-increments}
 0\le a_\ell,b_\ell\le\frac{\pi\beta}{8}.
\end{equation}
Direct integration of the $h^0$ equations in
Subsection~\ref{app:leading-factor-cycle} gives the following rank
coefficients, where \(\omega\) is the angular advance within the
indicated substep:
\begin{equation}\label{eq:reduction-rank-coefficients}
\begin{array}{c|c|c}
\text{substep}&\text{range}&\text{rank coefficient}\\
\hline
K_1&0\le\omega\le a_\ell&\cos\omega\\
S_1&0\le\omega\le a_\ell&2\cos a_\ell-\cos\omega\cos(a_\ell-\omega)\\
L&0\le\omega\le a_\ell+b_\ell&\cos(\omega-a_\ell)\\
S_2&0\le\omega\le b_\ell&2\cos b_\ell-\cos\omega\cos(b_\ell-\omega)\\
K_2&0\le\omega\le b_\ell&\cos(b_\ell-\omega).
\end{array}
\end{equation}
The K and L coefficients are bounded below by
\(\cos(\pi\beta/8)>0\), and the two S coefficients satisfy
\begin{equation}\label{eq:reduction-rank-margin}
 2\cos(\pi\beta/8)-\frac{1+\cos(\pi\beta/8)}2
 =\frac{3\cos(\pi\beta/8)-1}{2}>0.
\end{equation}

\subsection{Rank preservation after switch-off}
\label{app:post-switch-rank}

\begin{lemma}\label{lem:post-switch-rank}
When the right-hand side is $M_\eta$, every exact K and L subproblem
preserves column rank and every exact S subproblem preserves
nonsingularity.
\end{lemma}

\begin{proof}
For a K subproblem, set
\(B_V:=V^\top\Omega V\) and
\(\sigma_K(t):=\int_{t_0}^tq_\eta(K(s)V^\top)\,ds\).  Then
\begin{align}
 \dot K
 &=M_\eta(KV^\top)V
 =q_\eta(KV^\top)(\Omega K-KB_V),
 \label{eq:post-reduction-k}\\
 K(t)&=e^{\sigma_K(t)\Omega}K(t_0)e^{-\sigma_K(t)B_V}.
 \notag
\end{align}
For an S subproblem, set \(B_U:=U^\top\Omega U\) and
\(\sigma_S(t):=\int_{t_0}^tq_\eta(US(s)V^\top)\,ds\).  Then
\begin{equation}\label{eq:post-reduction-s}
 S(t)=e^{-\sigma_S(t)B_U}S(t_0)e^{\sigma_S(t)B_V}.
\end{equation}
The L formula is identical to the K formula with
\((K,V,B_V)\) replaced by \((L,U,B_U)\).  Since
\(\Omega,B_U,B_V\) are skew-symmetric, every exponential above is
orthogonal.  The asserted ranks are therefore preserved.
\end{proof}

\subsection{A uniform rank neighborhood}
\label{app:uniform-neighborhood}

For \(h>0\), define
\begin{equation}\label{eq:appendix-block-start-chart}
 S_h(s_{11},s_{12},s_{21},\Delta_S):=
 \begin{pmatrix}
 s_{11}&s_{12}\\
 s_{21}&(s_{12}s_{21}+h^3\Delta_S)/s_{11}
 \end{pmatrix}.
\end{equation}
For orthonormal \(U,V\in\mathbb R^{3\times2}\), set
\begin{equation}\label{eq:appendix-fixed-neighborhood}
\begin{aligned}
 d_{\mathrm{fac}}:={}&\norm{U-\mathcal B(0)}_{\F}
 +\norm{V-\mathcal B(0)}_{\F}\\
 &+|s_{11}-1|+|s_{12}|+|s_{21}|+|\Delta_S-1|.
\end{aligned}
\end{equation}

\begin{lemma}[Uniform rank neighborhood]
\label{lem:desingularized-rank-neighborhood}
There exist \(\rho\in(0,1/2)\) and \(h_*>0\), independent of
\(\eta\in\{+1,-1\}\), such that, whenever \(d_{\mathrm{fac}}\le\rho\) and
\(0<h\le h_*\), every periodic four-step block starting from the factor
triple
\[
 \bigl(U,S_h(s_{11},s_{12},s_{21},\Delta_S),V\bigr)
\]
satisfies, throughout its corresponding factor substeps,
\begin{equation}\label{eq:appendix-uniform-denominators}
\begin{aligned}
 |s_{11}|&\ge\frac12,
 &|\Delta_S|&\ge\frac{3\cos(\pi/8)-1}{4},\\
 \norm{\boldsymbol k_1}_2,\norm{\boldsymbol\lambda_1}_2
 &\ge\frac12,\\
 \norm{\boldsymbol w_K\times\boldsymbol k_1}_2,
 \norm{\boldsymbol w_L\times\boldsymbol\lambda_1}_2
 &\ge\frac12\cos(\pi/8).
\end{aligned}
\end{equation}
Consequently, the complete block is rank-admissible.  The same
conclusion holds for the four-step switch-off block.
\end{lemma}

\begin{proof}
If \(d_{\mathrm{fac}}\le\rho\), then
\begin{equation}\label{eq:appendix-s11-lower-bound}
 |s_{11}|\ge1-\rho>\frac12,
 \qquad
 S_h=
 \begin{pmatrix}s_{11}\\s_{21}\end{pmatrix}
 \begin{pmatrix}1&s_{12}/s_{11}\end{pmatrix}
 +h^3\begin{pmatrix}0&0\\0&\Delta_S/s_{11}\end{pmatrix}.
\end{equation}
Equations \eqref{eq:appendix-K-column-equation},
\eqref{eq:appendix-normalized-K}, and
\eqref{eq:appendix-normalized-K-parallel}, together with
\eqref{eq:appendix-QR-reconstruction}, form a closed system for
\((\boldsymbol k_1,p_K,\boldsymbol w_K)\).  Its QR transfer is
\eqref{eq:appendix-signed-QR-basis}--
\eqref{eq:appendix-signed-QR-determinant}; its only denominators are
\begin{equation}\label{eq:appendix-K-denominators}
 \norm{\boldsymbol k_1}_2,
 \qquad
 \norm{\boldsymbol w_K\times\boldsymbol k_1}_2.
\end{equation}
Equations \eqref{eq:appendix-S-equation},
\eqref{eq:appendix-core-coordinate-odes}, and
\eqref{eq:appendix-normalized-S}, together with
\eqref{eq:appendix-core-reconstruction}, are closed in
\((s_{11},s_{12},s_{21},\Delta_S)\); their only denominator is
\(|s_{11}|\).  The L equations are the same closed construction with
\((\boldsymbol k_i,p_K,\boldsymbol w_K)\) replaced by
\((\boldsymbol\lambda_i,p_L,\boldsymbol w_L)\).

The bounds proved in Subsection~\ref{app:positive-rank-margins} give,
along every $h=0$ desingularized path,
\begin{equation}\label{eq:appendix-reference-rank-margin}
\begin{aligned}
 \norm{\boldsymbol k_1}_2
 &=\norm{\boldsymbol\lambda_1}_2=|s_{11}|=1,\\
 \norm{\boldsymbol w_K\times\boldsymbol k_1}_2,
 \ \norm{\boldsymbol w_L\times\boldsymbol\lambda_1}_2
 &\ge\cos(\pi/8),
 &|\Delta_S|&\ge\frac{3\cos(\pi/8)-1}{2}.
\end{aligned}
\end{equation}

The union of the twenty $h=0$ desingularized substep paths is compact.  On the open
set where the quantities in \eqref{eq:appendix-K-denominators} and
\(|s_{11}|\) are nonzero, the desingularized differential equations
and the QR and core transfers are continuous.  Starting with
\eqref{eq:appendix-reference-rank-margin}, continuous dependence
preserves half of every lower bound on the first substep.  Applying the
smooth transfer gives the initial data for the next substep.  Repeating
this argument through the finite list \eqref{eq:factor-transfers}
chooses one \(\rho\) and one \(h_*\) for which
\eqref{eq:appendix-uniform-denominators} holds through all twenty
substeps.

For the switch-off block, Subsection~\ref{app:switch-off-margins}
establishes the reference-path margins
\eqref{eq:reduction-rank-coefficients}--
\eqref{eq:reduction-rank-margin}.  In particular,
\begin{equation}\label{eq:appendix-switch-off-margin-comparison}
 \cos(\pi\beta/8)>\frac12\cos(\pi/8),\qquad
 \frac{3\cos(\pi\beta/8)-1}{2}
 >\frac{3\cos(\pi/8)-1}{4}.
\end{equation}
The same finite-step argument therefore gives
\eqref{eq:appendix-uniform-denominators} for the switch-off block.

Finally,
\begin{equation}\label{eq:appendix-rank-conclusion}
 \boldsymbol k_1\times\boldsymbol k_2=h^3\boldsymbol w_K,
 \qquad \det S=h^3\Delta_S,
 \qquad \boldsymbol\lambda_1\times\boldsymbol\lambda_2=h^3\boldsymbol w_L.
\end{equation}
For \(h>0\), \eqref{eq:appendix-uniform-denominators} therefore implies
\(\operatorname{rank}K=\operatorname{rank}L=2\) and \(\det S\ne0\).
\end{proof}

\subsection{Uniform block and projector expansions}
\label{app:uniform-block-expansions}

For the step indexed by \(\ell\), let \(V_0\) be its input orthonormal
basis for the matrix row space,
\(U_1\) the basis after \(K_1\), and \(V_1\) the basis after \(L\).
Define
\begin{equation}\label{eq:appendix-projector-definitions}
 P^h_{\ell,\eta}:=V_0V_0^\top,
 \qquad
 Q^h_{\ell,\eta}:=U_1U_1^\top,
 \qquad
 P^h_{\ell+1,\eta}:=V_1V_1^\top.
\end{equation}

\begin{lemma}[Uniform block and projector expansions]
\label{lem:desingularized-block-map}
Let \(\rho\) and \(h_*\) be as in
Lemma~\ref{lem:desingularized-rank-neighborhood}.  The following three
conclusions hold for \(d_{\mathrm{fac}}\le\rho\).

First, the complete periodic four-step factor output after the gauge
\eqref{eq:appendix-gauge-map}, its numerical matrix endpoint, and the
corresponding full-flow endpoint extend jointly \(C^3\) to
\(0\le h\le h_*\) as functions of
\((h,U,V,s_{11},s_{12},s_{21},\Delta_S)\).  Using fixed input and output
charts near $\mathcal B(0)$ and a fixed finite family of phase-local
Stiefel charts for the intermediate bases, their derivatives through
total order three are bounded by a constant \(C\), independently of
\(\eta\).

Second, for the initial tuple
\begin{equation}\label{eq:appendix-initial-tuple}
 (U,V,s_{11},s_{12},s_{21},\Delta_S)
 =(\mathcal B(0),\mathcal B(0),1,0,0,1),
\end{equation}
the projectors in \eqref{eq:appendix-projector-definitions} satisfy
\begin{align}
 P^{h}_{\ell,\eta}
 &=E_{11}+W(\ell)+hP^{(1)}_\ell+O(h^2),\notag\\
 Q^{h}_{\ell,\eta}
 &=E_{11}+W(\ell+\tfrac12)+hQ^{(1)}_\ell+O(h^2),\notag\\
 P^{h}_{\ell+1,\eta}
 &=E_{11}+W(\ell+1)+hP^{(1)}_{\ell+1}+O(h^2),
 \label{eq:appendix-projector-jets}
\end{align}
where the three linear coefficients are independent of \(\eta\).
Third, for every such tuple,
\begin{equation}\label{eq:appendix-nearby-projectors}
\begin{aligned}
 \norm{P^h_{\ell,\eta}-E_{11}-W(\ell)}_{\F}
 &\le C(d_{\mathrm{fac}}+h),\\
 \norm{Q^h_{\ell,\eta}-E_{11}-W(\ell+\tfrac12)}_{\F}
 &\le C(d_{\mathrm{fac}}+h),\\
 \norm{P^h_{\ell+1,\eta}-E_{11}-W(\ell+1)}_{\F}
 &\le C(d_{\mathrm{fac}}+h).
\end{aligned}
\end{equation}
\end{lemma}

\begin{proof}
Let $\mathcal K$ be a closed tube around the union of the twenty
$h=0$ desingularized substep paths.  By
\eqref{eq:appendix-reference-rank-margin}, the tube can be chosen inside
the open set on which $|s_{11}|$, $\norm{\boldsymbol k_1}_2$,
$\norm{\boldsymbol w_K\times\boldsymbol k_1}_2$,
$\norm{\boldsymbol\lambda_1}_2$, and
$\norm{\boldsymbol w_L\times\boldsymbol\lambda_1}_2$ have one common
positive lower bound.  After decreasing $\rho$ and $h_*$ if necessary,
Lemma~\ref{lem:desingularized-rank-neighborhood} places every substep
trajectory with $d_{\mathrm{fac}}\le\rho$ and $0\le h\le h_*$ in $\mathcal K$.

After the reconstructions \eqref{eq:appendix-QR-reconstruction}--
\eqref{eq:appendix-core-reconstruction}, the K, S, and L coordinate
equations contain only nonnegative powers of $h$.  Their vector fields
are therefore $C^3$ on a neighborhood of
$[0,h_*]\times\mathcal K$, with derivatives through order three bounded
uniformly in $\eta$.  The signed QR maps
\eqref{eq:appendix-signed-QR-basis}--
\eqref{eq:appendix-signed-QR-determinant} have the same property.  The
remaining transfers in \eqref{eq:factor-transfers} are transposition
and matrix multiplication, and the gauge
\eqref{eq:appendix-gauge-map} is linear.  Thus all vector fields and
transfer maps have uniformly bounded derivatives through order three
on the same tube.  Differentiating each fixed-interval integral
equation up to three times in $h$ and the initial coordinates gives
linear variational equations with lower-order inhomogeneous terms.
Gronwall's inequality bounds them recursively, and finite composition
over the twenty substeps and the final gauge preserves one common bound.

The corresponding full solution on one periodic block satisfies
\begin{equation}\label{eq:appendix-full-flow-parameter-equation}
 \frac{dA}{d\theta}=hM_\eta(A)+h^3W'(\theta),
 \qquad 0\le\theta\le4.
\end{equation}
The same integral-equation argument gives the uniform \(C^3\) bound for
this full-flow endpoint.

At \eqref{eq:appendix-initial-tuple}, the constant terms of the
projectors are
\[
 P^0_{\ell,\eta}=E_{11}+W(\ell),
 \qquad
 Q^0_{\ell,\eta}=E_{11}+W(\ell+\tfrac12),
 \qquad
 P^0_{\ell+1,\eta}=E_{11}+W(\ell+1).
\]
Along this $h=0$ desingularized path, the represented matrix is
\(E_{11}\), and
\begin{equation}\label{eq:appendix-equal-first-jets}
 q_+(E_{11})=q_-(E_{11})=\frac1{\sqrt2}.
\end{equation}
Since every occurrence of \(q_\eta\) in the desingularized equations
is multiplied by \(h\),
\[
 \left.\partial_h\bigl(hq_\eta(X(h))\bigr)\right|_{h=0}
 =q_\eta(E_{11});
\]
terms containing $Dq_\eta$ enter only at the next order.  Therefore
\eqref{eq:appendix-equal-first-jets} gives
\begin{equation}\label{eq:appendix-projector-derivative-equality}
\begin{aligned}
 \left.\partial_hP^h_{\ell,+}\right|_{h=0}
 &=\left.\partial_hP^h_{\ell,-}\right|_{h=0}
 =:P^{(1)}_\ell,\\
 \left.\partial_hQ^h_{\ell,+}\right|_{h=0}
 &=\left.\partial_hQ^h_{\ell,-}\right|_{h=0}
 =:Q^{(1)}_\ell,\\
 \left.\partial_hP^h_{\ell+1,+}\right|_{h=0}
 &=\left.\partial_hP^h_{\ell+1,-}\right|_{h=0}
 =:P^{(1)}_{\ell+1}.
\end{aligned}
\end{equation}
Taylor's formula applied to
\eqref{eq:appendix-projector-derivative-equality} proves
\eqref{eq:appendix-projector-jets}.  The mean-value theorem applied to
the jointly \(C^1\) projector maps gives
\eqref{eq:appendix-nearby-projectors}.
\end{proof}

\bibliographystyle{siamplain}
\bibliography{references}

\end{document}